\documentclass[12pt]{amsart}
\usepackage{amsmath}
\usepackage{amsthm}
\usepackage{amssymb}
\usepackage{amscd}
\usepackage{amsfonts}
\usepackage{amsbsy}
\usepackage{epsfig}
\usepackage{graphicx}
\usepackage{tikz}
\usepackage{tikz-cd}
\usepackage{appendix}

\newtheorem{theorem}{Theorem}[section]
\newtheorem{corollary}[theorem]{Corollary}
\newtheorem{proposition}[theorem]{Proposition}
\newtheorem{observation}[theorem]{Observation}
\newtheorem{lemma}[theorem]{Lemma}
\newtheorem{remark}[theorem]{Remark}
\newtheorem{question}[theorem]{Question}
\newtheorem{definition}[theorem]{Definition}
\newtheorem{example}[theorem]{Example}

\newtheorem{problem}{Problem}

\def\ie{{\em i.e.,} }
\def\eg{{\em e.g.} }

\newfont\bbf{msbm10 at 12pt}
\def\eps{\varepsilon}
\def\phi{\varphi}
\def\R{{\mathbb R}}

\def\N{{\mathbb N}}
\def\Z{{\mathbb Z}}

\def\E{{\mathcal E}}
\def\F{{\mathcal F}}

\def\orb{\mbox{\rm orb}}
\def\theta{\vartheta}

\def\eps{\varepsilon}

\def\Int{\mbox{\rm Int}}

\def\diam{\mbox{\rm diam}}
\def\Fix{\mbox{\rm Fix}}

\begin{document}

\title{Inhomogeneities in chainable continua}
\author{Ana Anu\v si\'c, Jernej \v Cin\v c}

\address[A.\ Anu\v{s}i\'c]{Departamento de Matem\'atica Aplicada, IME-USP, Rua de Mat\~ao 1010, Cidade Universit\'aria, 05508-090 S\~ao Paulo SP, Brazil}
\email{anaanusic@ime.usp.br}
\address[H.\ Bruin]{Faculty of Mathematics, University of Vienna, Oskar-Morgenstern-Platz 1, A-1090 Vienna, Austria}
\email{henk.bruin@univie.ac.at}
\address[J.\ \v{C}in\v{c}]{AGH University of Science and Technology, Faculty of Applied Mathematics, 
	al.\ Mickiewicza 30, 30-059 Krak\'ow, Poland. -- and -- National Supercomputing Centre IT4Innovations, 
	Division of the University of Ostrava, Institute for Research and Applications of Fuzzy Modeling, 
	30. dubna 22, 70103 Ostrava, Czech Republic}
\email{jernej@agh.edu.pl}
\thanks{We thank Henk Bruin for helpful remarks and Jan P. Boro\'nski for suggesting to consider circle like inverse limits. AA was supported by grant 2018/17585-5, S\~ao Paulo Research Foundation (FAPESP). J\v{C} was supported by the FWF Schr\"odinger Fellowship stand-alone project J4276-N35. J\v{C} was also partially supported by University of Ostrava grant IRP201824 “Complex topological structures”.}
\date{\today}

\subjclass[2010]{37B45, 37E05, 54H20}
\keywords{inverse limit space, endpoint, folding point, inhomogeneity, persistent recurrence, zigzag}

\maketitle

\begin{center}
	(with Appendix~A by Henk Bruin)
\end{center}

\begin{abstract}
	We study a class of chainable continua which contains, among others, all inverse limit spaces generated by a single interval bonding map which is piecewise monotone and locally eventually onto. Such spaces are realized as attractors of non-hyperbolic surface homeomorphisms. Using dynamical properties of the bonding map, we give conditions for existence of endpoints,  characterize the set of local inhomogeneities, and determine when it consists only of endpoints. As a side product we also obtain a characterization of arcs as inverse limits for piecewise monotone bonding maps, which is interesting in its own right.
\end{abstract}

\section{Introduction}\label{sec:intro}
Williams showed \cite{Wil1,Wil2} that expanding uniformly hyperbolic attractors can be represented as inverse limits on branched one-manifolds and the dynamics can be completely understood in the terms of the underlying one-dimensional map. Moreover, such attractors are locally homeomorphic to the product of a Cantor set and an (open) arc. Thus, in order to understand the properties of non-hyperbolic systems, we need to focus on those which have inhomogeneous global attractors. The simplest, yet topologically interesting such systems are
inverse limits $\underleftarrow{\lim}(I, f)$, with a continuous bonding map $f\colon I\to I$. It is well-known that every $\underleftarrow{\lim}(I, f)$ can be embedded in the plane as a global attractor of a planar homeomorphism, which is conjugate to the natural extension of $f$ on its global attractor, see \cite{BaMa}. Inverse limits of interval maps contain points which are locally not homeomorphic to the Cantor set of arcs, called {\em folding points}. The term folding point was introduced by Raines in \cite{Raines} and the 
name emphasizes the occurrence of ``folds" in arbitrary small neighbourhoods of such a point, see Figure~\ref{fig:uils}. Our goal is to understand the structure and properties of the set of folding points in terms of the dynamics of the bonding map $f$.

A {\em continuum} is a nonempty compact connected metric space. It is called {\em chainable} (or {\em arc-like}), if it admits an $\eps$-mapping onto
the interval $[0,1]$ for every $\eps>0$. Every chainable continuum can be represented as an inverse limit on intervals, possibly with varying bonding maps (see \eg \cite{InMah}).
In this paper we study topological properties of a wide class of chainable continua $\underleftarrow{\lim}\{I,f\}$, where $f\colon I\to I$ is a {\em long-zigzag} map. 

\begin{definition}\label{def:longzigzag}
	A map $f\colon I\to I$ is called {\em long-zigzag} 
	if there exists $\eps>0$ such that for every $n\in\N_0$ and every interval  $[a,b]\subset I$ such that $f^n([a, b])=[f^n(a), f^n(b)]$ (or $f^n([a,b])=[f^n(b), f^n(a)]$), $f^n(x)\not\in\{f^n(a),f^n(b)\}$ for every $x\in (a,b)$, and $|f^n(a)-f^n(b)|<\eps$, it holds that $f^n$ maps $[a,b]$ homeomorphically onto $f^n([a,b])$.
\end{definition}

Intuitively, long-zigzag maps are interval maps for which the lengths of possible {\em zigzags} (see Figure~\ref{fig:zigzag}) in all iterates are bounded from below.
Specifically, a map $f$ will be called {\em zigzag-free} if it has a property that whenever $f([a,b])=[f(a),f(b)]$ or $f([a,b])=[f(b),f(a)]$ and $f(x)\not\in \{f(a),f(b)\}$ for all $x\in(a,b)$, then $f|_{[a,b]}$ is a homeomorphism, see Definition~\ref{def:zigzagfree}. 
Every zigzag-free map, and this includes all unimodal maps (see Lemma 2.1. in \cite{ABC1}), is long-zigzag.
On the other hand the interval maps which produce the {\em pseudo-arc}  (\eg Henderson's map \cite{Henderson}) are not long-zigzag maps. 
 
 We show that the assumption of a map being long-zigzag is actually very mild - if interval map $f$ has finitely many critical points $\{c_1, \ldots, c_n\}$ such that if $f^k|_{[c_i,c_{i+1}]}$ is one-to-one for every $k\in\N$, then the length of branches $|f^k([c_i,c_{i+1}])|$ is bounded from below, then $f$ is long-zigzag. Specifically, every piecewise monotone $f$ which is locally eventually onto is long-zigzag, see Corollary~\ref{thm:leo+fpwl=lzzb}. We say that $f$ is \emph{locally eventually onto (leo)}, if every open interval in $I$ is eventually mapped onto the whole $I$. To assume leo is also not a severe restriction in our context since any interval map with finitely many critical points and without restrictive intervals, 
 periodic attractors or wandering intervals is conjugate to a piecewise linear leo map, or semi-conjugate otherwise 
 (see \cite{MT} or Theorem 8.1 from \cite{dMvS}). 
It is easy to see that there also exist long-zigzag interval maps with infinitely many critical points.
Moreover, the space of long-zigzag interval maps is dense in the $C^k$-topology for every $k\in \N$ in the space of all 
continuous interval maps. This follows since long-zigzag maps include long-branched interval maps and the later contain 
the set of hyperbolic maps. Hyperbolic maps are dense in the set of surjective interval maps in the $C^k$-topology for every $k\in\N$, 
see \cite{dMvS, KSvS} for details.
  
 Our study is motivated by results on the structure of local inhomogeneities of unimodal inverse limit spaces, see \eg \cite{BM,Raines,Br1,AABC}, and by the lack of basic tools to study inhomogeneities of indecomposable one-dimensional continua beyond the unimodal setting.
 
 An interval map $f$ is called {\em unimodal} 
 if there exists a unique critical point $c\in (0,1)$. Unimodal inverse limits present a simplified models for global attractors of H\'enon maps $H_{a,b}\colon \R^2\to\R^2$, $H_{a,b} = (1-ax^2+by, x)$ (sometimes they are indeed an actual model, see \cite{BaHo}). They have recently generated a significant research interest, both in continuum theory, culminating with the proof of the Ingram conjecture in \cite{BBS}, and in dynamical systems, by the works of Boyland, de Carvalho and Hall \cite{3G,3G1,3G3,3G2} (from both topological and measure theoretical perspective).
 When studying the local structure of unimodal inverse limits, it turns out that one of the key properties is that unimodal maps are long-zigzag. Using only analytic observations for interval maps (specifically avoiding symbolic arguments used in the unimodal setting), we prove that the following results on unimodal inverse limits hold in a much wider setting.

{\bf (1) The properties of endpoints.} If $f\colon I\to I$ is unimodal, then every endpoint in $\underleftarrow{\lim}\{I,f\}$ is an endpoint of every arc which contains it (Bruin \cite{Br1}).

A point $x$ in a chainable continuum $K$ is called an {\em endpoint}, if for every two subcontinua $x\in A, B\subset K$ it holds that $A\subset B$ or $B\subset A$. A point $x$ in a continuum $K$ is called an {\em $L$-endpoint} if $x$ is an endpoint of every arc in which it is contained (this definition was given by Lelek~\cite{Le} for the study of dendroids).

\begin{definition}
	Let $X=\underleftarrow{\lim}\{I, f_i\}$ and $x=(x_0, x_1, \ldots)\in X$. For $i\in\N_0$ we define the {\em $i$-basic arc} $A_i(x)$ as the maximal arc in $X$ such that $x\in A_i(x)$ and $\pi_i|_{A_i(x)}$ is one-to-one (can be degenerate).
\end{definition}

 Bruin~\cite{Br1} (see also Brucks and Diamond \cite{BrDi}) introduced this notion of a basic arc and then proves that in unimodal inverse limits a point $x$ is an endpoint of $K$ if and only if it is an endpoint of its $i$-basic arc for every $i\geq 0$ (here we call that a {\em $B$-endpoint}). 
 It follows that every endpoint in $\underleftarrow{\lim}\{I,f\}$ is an $L$-endpoint, if $f\colon I\to I$ is unimodal.\\
   The following theorem generalizes Bruin's result for the class of long-zigzag leo maps. 

{\bf Theorem 4.7.}
	If $f\colon I\to I$ is a long-zigzag leo map, then the sets of $B$-endpoints, $L$-endpoints and endpoints in $\underleftarrow{\lim}\{I,f\}$ are the same.  

When studying $L$-endpoints, a natural question which arises is how to characterize interval inverse limits $\underleftarrow{\lim}\{I,f_n\}$ which are homeomorphic to an arc. In Appendix~\ref{appB} we provide a characterization of an arc as an inverse limit with a single bonding map which has a finite set of critical points. That is a fundamental result, interesting also on its own. The characterization in a more general setting is still outstanding.

{\bf (2) Dynamical characterization of folding points.} Let $C$ denote the {\em set of critical points} of a 
continuous interval map $f$ including the points $0$ and $1$. 
The {\em omega limit set $\omega(C)$} of $C$ is the set of all limit points of sequences $(f^n(c))_{n\in\N}$, for all $c\in C$. For a large class of interval maps $f$, which contains unimodal maps, a folding point of $\underleftarrow{\lim}\{I,f\}$ can be characterized as a point for which all projections are in $\omega(C)$ (Raines \cite{Raines}). The class of maps from Raines' theorem in contained in a class of long-zigzag maps, see Corollary~\ref{cor:Raines}.

 We generalize Raines' characterization to long-zigzag interval maps $f$. For every $n\in\N$ 
 we denote the {\em $n$-th natural coordinate projection} of $\underleftarrow{\lim}\{I, f\}$ to $I$ by $\pi_n$.

{\bf Theorem 5.1.}
	Assume $f\colon I\to I$ is long-zigzag map. Then $x\in\underleftarrow{\lim}\{I, f\}$ is a folding point if and only if $\pi_n(x)\in\omega(C)$ for every $n\in\N_0$.

{\bf (3) Retractability and endpoints.}  If $f$ is unimodal, but not conjugate to the full tent map, then the set of endpoints in $\underleftarrow{\lim}\{I,f\}$ is equal to the set of folding points if and only if $c$ is persistently recurrent, see \cite{AABC}.

The notion of persistent recurrence was first introduced in \cite{BL} in connection with the existence 
of wild attractors of unimodal interval maps and was used in the construction of wild Cantor set attractors in \cite{BKNS}. 
We generalize the notion of persistent recurrence and characterize the class of long-zigzag inverse limits for which the set of folding points equals the set of endpoints. We use both Theorem~\ref{thm:llzzb} and Theorem~\ref{thm:folding} in the proof of the theorem below. 


\begin{definition}
    The map $f$ is {\em non-retractable along $\omega(C)$} if for every
	backward orbit $x=(x_0, x_1, x_2, \ldots)$ in $\omega(C)$ and every open interval $U\subset I$ such that $x_0\in U$, the {\em pull-back} of $U$ along $x$ is not monotone.
\end{definition}

Definition~\ref{def:pullback} explains precisely what we mean by a  monotone pull-back along $x$.
Note that the definition does not require recurrence of critical points. For example, the critical set $C=\{0,1/2,1\}$ of the full tent map $T_2(x)=\min\{2x,2(1-x)\}$, $x\in I$, is non-retractable, yet $c=1/2$ is not recurrent. Due to the lack of recurrence in this more general setting, we avoid the use of the term ''persistent recurrence''.

{\bf Theorem 6.3.}
	Let $f\colon I\to I$ be a long-zigzag leo map with the critical set $C$. Then for $X=\underleftarrow{\lim}\{I, f\}$ it holds that every folding point is an endpoint if and only if $f$ is non-retractable along $\omega(C)$.
	
Since it is in general hard to see from the bonding map if the map is non-retractable or not, we provide in Appendix~\ref{appA} jointly with Henk Bruin a characterization of interval maps that are retractable through their dynamical properties.
	
{\bf (4) Effect of critical recurrence on the existence of endpoints.} If $f\colon I\to I$ is unimodal, then $\underleftarrow{\lim}\{I,f\}$ contains an endpoint if and only if $c$ is recurrent, \ie $c\in\omega(c)$. Moreover, if the orbit of $c$ is infinite, then $\underleftarrow{\lim}\{I,f\}$ has uncountably many endpoints, see \cite{BM} and \cite{Br1}. 

We show that the similar result holds true for long-zigzag leo maps $f$, assuming that at least one critical point recurs on itself (\ie $c\in\omega(c)$).
We emphasize that we avoid using the technical characterization of endpoints from Theorem 1.4 of \cite{BM}, which makes our proof more accessible, even in the unimodal setting.

{\bf Theorem 7.3.}
If $f$ is long-zigzag leo map, and there exists $c\in C$ such that $c\in\omega(c)$, then there exists an endpoint in $\underleftarrow{\lim}\{I,f\}$. If additionally $\orb(c)$ is infinite, then $\underleftarrow{\lim}\{I,f\}$ has uncountably many endpoints.	

We note that the mentioned results present foundations for more general theory in the setting of graph inverse limits (and thus $1$-dimensional continua in general), once the definitions given in this paper are adjusted accordingly.
For example, the notion of a long-zigzag circle map $f\colon S^1\to S^1$ needs to be defined on a lift $F\colon\R\to\R$, and then the arguments in the proofs of four theorems stated above follow analogously. Therefore, the above four theorems hold in the setting of inverse limits with circle bonding maps as well. When underlying graphs have branching points, the definition of a folding point has to be given more precisely in order to give an applicable theory. 

Let us briefly outline the structure of the paper. After fixing notation and giving preliminaries we discuss zigzags in interval maps and show that every leo map with finitely many 
critical points is long-zigzag in Section~\ref{sec:zigzags}. In Section~\ref{sec:endpoints} we discuss different notions of endpoints, 
and finish our discussion with the proofs of propositions which imply Theorem~\ref{thm:llzzb}. 
In Sections~\ref{sec:fp},~\ref{sec:persistence} and ~\ref{sec:recurrence}, we prove Theorems~\ref{thm:folding}, ~\ref{thm:FisE} and ~\ref{thm:endpt} respectively. Finally, Appendix~\ref{appA} gives a characterization of retractability along $\omega(C)$ through properties of interval maps and Appendix~\ref{appB} provides a characterization of arc when the bonding map has finitely many critical points. Whenever possible, we include examples which show that the standing assumptions are indeed optimal. 

\section{Preliminaries and notation}\label{sec:prel}

The set of {\em natural numbers} is denoted by $\N$, and $\N_0=\N\cup\{0\}$.
The {\em interior} and the {\em closure} of a set $S$ will be denoted by $\Int(S)$ and $\overline{S}$, 
respectively. By $I$ we denote the {\em unit interval} $[0,1]$. A {\em map} $f\colon I\to I$ always 
means a continuous surjective function, and $f^n=f\circ\ldots\circ f$ is the $n$-th iterate of $f$. 
A point $c\in (0,1)$ is called a {\em critical point} of $f\colon I\to I$ if for every open $J\ni c$, 
$f|_J$ is not one-to-one. 
The set $C$ of all critical points of $f$ including $0$ and $1$ is called the {\em critical set}. 
An interval map is called {\em piecewise monotone}, if it has finitely many critical points. Given $f\colon I\to I$, an {\em orbit of a point} $x\in I$ is the set $\orb(x)=\{f^n(x): n\geq 0\}$, and an {\em orbit of a set} $S\subset I$ is $\orb(S)=\{\orb(x): x\in S\}$. The {\em omega-limit set} of a set $S\subset I$, denoted by $\omega(S)$, is the set of all limit points of $\orb(S)$. 


A {\em continuum} is a non-empty, compact, connected, metric space. In the definition of inverse limit space we can assume without loss of generality that the bonding functions are surjective.
Given a sequence of maps 
$f_i\colon I\to I$, $i\in\N$, the inverse limit is given by:
$$\underleftarrow{\lim}\{I, f_i\}=\{(x_0, x_1, \ldots): f_i(x_i)=x_{i-1}, i\in\N\}\subset I^{\N_0}.$$
Equipped with the {\em product topology}, $\underleftarrow{\lim}\{I, f_i\}$ is a continuum, and it is {\em chainable}, 
\ie for every $\eps>0$ there exists an $\eps$-mapping onto an interval. 
Moreover, every chainable continuum is an inverse limit on intervals, see \cite{InMah}. The {\em coordinate projections} are defined 
by $\pi_i\colon \underleftarrow{\lim}\{I, f_i\}\to I$, $\pi_i((x_0, x_1, \ldots))=x_i$, for $i\in\N_0$ and they are all continuous. 
Also, for $i\in\N_0$, an {\em $i$-box} is the set of the form $\pi_i^{-1}(U)$, where $U\subset I$ is an open set. 
The set of all $i$-boxes, $i\geq 0$, forms a basis for the topology of $\underleftarrow{\lim}\{I, f_i\}$. 

If $f_i=f$ for all $i\in\N$, we denote the inverse limit by $X=\underleftarrow{\lim}\{I, f\}$. In that case there exists a 
homeomorphism $\hat f\colon X\to X$, given by $$\hat f((x_0, x_1, \ldots)):=(f(x_0), f(x_1), f(x_2), \ldots)=(f(x_0), x_0, x_1, \ldots).$$
 It is called the {\em natural extension} of $f$ (or sometimes {\em shift homeomorphism}). If $f$ is {\em unimodal} 
 then $X$ is called a {\em unimodal inverse limit space}.
 
\section{Zigzags in interval maps}\label{sec:zigzags}

In this section we define the basic notions that we need in the rest of the paper and prove in Corollary~\ref{thm:leo+fpwl=lzzb} that every piecewise monotone leo map is long-zigzag.

\begin{definition}\label{def:zigzagfree}
	A map $f\colon I\to I$ is called \em{zigzag-free} if for every interval $[a, b]\subset I$  such that $f([a,b]))=[f(a), f(b)]$ (or $f([a,b])=[f(b),f(a)]$), and $f(x)\not\in\{f(a),f(b)\}$ for every $x\in (a,b)$, it holds that $[a,b]$ is mapped homeomorphically onto $f([a,b])$ (see examples of maps which are not zigzag-free in Figure~\ref{fig:zigzag}).
\end{definition}

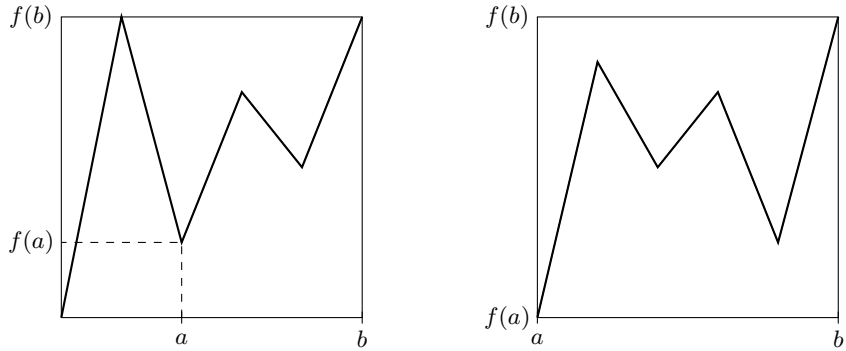
\begin{figure}[!ht]
	\centering
	\begin{tikzpicture}[scale=4]
	\draw (0,0)--(1,0)--(1,1)--(0,1)--(0,0);
	\draw [thick] (0,0)--(1/5,1)--(2/5,1/4)--(3/5,3/4)--(4/5,1/2)--(1,1);
	\draw (2/5,-0.02)--(2/5,0.02);
	\node at (2/5, -0.07) {\scriptsize $a$};
	\node at (-0.1, 1/4) {\scriptsize $f(a)$};
	\node at (-0.1, 1) {\scriptsize $f(b)$};
	\draw[dashed] (2/5,0)--(2/5,1/4)--(0,1/4);
	\draw (1,-0.02)--(1,0.02);
	\node at (1, -0.07) {\scriptsize $b$};
	\draw (2/5,1/4)--(3/5,3/4)--(4/5,1/2)--(1,1);
	\end{tikzpicture}
	\hspace{1cm}
	\begin{tikzpicture}[scale=4]
	\draw (0,0)--(1,0)--(1,1)--(0,1)--(0,0);
	\draw[thick] (0,0)--(0.2,0.85)--(0.4,0.5)--(0.6,0.75)--(0.8,0.25)--(1,1);
	\draw (0,-0.02)--(0,0.02);
	\node at (0, -0.07) {\scriptsize $a$};
	\node at (-0.1,0) {\scriptsize $f(a)$};
	\node at (-0.1, 1) {\scriptsize $f(b)$};
	\draw (1,-0.02)--(1,0.02);
	\node at (1, -0.07) {\scriptsize $b$};
	\draw (0,0)--(0.2,0.85)--(0.4,0.5)--(0.6,0.75)--(0.8,0.25)--(1,1);
	\end{tikzpicture}
	\caption{Maps $f$ which are not zigzag-free.}
	\label{fig:zigzag}
\end{figure}


Note that there is no smoothness condition regarding the function $f$ in the definition above, however for simplicity we will draw all pictures piecewise linear.
In the next two results we need the following observation.

\begin{observation}\label{obs:compzz}
	For maps $f,g\colon I\to I$ write $F=f\circ g$ and assume $a<b\in I$ are such that $F([a,b])\in\{[F(a),F(b)], [F(b),F(a)]\}$ and $F(x)\not\in\{F(a),F(b)\}$ for all $x\in(a,b)$. Then also $g([a,b])\in\{[g(a), g(b)], [g(b), g(a)]\}$ and $g(x)\not\in\{g(a),g(b)\}$ for all $x\in(a,b)$. The same holds for $f$, \ie $f([g(a),g(b)])\in\{[f\circ g(a), f\circ g(b)]\}$, and $f(x)\not\in\{f\circ g(a), f\circ g(b)\}$ for all $x\in (g(a),g(b))$. Thus $f|_{[g(a),g(b)]}$ is either monotone, or $f$ has at least two critical points in $(g(a),g(b))$, in which case $|F(a)-F(b)|\geq\min\{|f(x)-f(y)|: x\neq y \text{ neighbouring critical points of $f$} \}$. 
\end{observation}

The last statement in the observation above will be used later in Theorem~\ref{thm:lzzbcond}.

\begin{lemma}\label{lem:itzigzag}
	If $f$ is zigzag-free, then $f^n$ is zigzag-free for every $n\in\N$.
\end{lemma}
\begin{proof}
	We use Observation~\ref{obs:compzz} in the case when $g=f$.
	Assume that $0\leq a<b\leq 1$ are such that $f^n([a,b])=[f^n(a), f^n(b)]$ or  $f^n([a,b])=[f^n(b), f^n(a)]$                                                                                                                   and $f^n(x)\not\in\{f^n(a),f^n(b)\}$ for every $x\in(a,b)$. 
	Then $f([a,b])=[f(a),f(b)]$ or $f([a,b])=[f(b),f(a)]$, and $f(x)\not\in\{f(a),f(b)\}$ also, so since $f$ is zigzag-free, $f|_{[a,b]}$ is monotone. The same holds for $f^{i+1}|_{[f^i(a),f^i(b)]}$ for every $i\in\{1,\ldots, n-2\}$. It follows that $f^n|_{[a,b]}$ is one-to-one, so $f^n$ is zigzag-free. 
\end{proof}

\begin{definition}\label{def:longzigzag}
	A map $f\colon I\to I$ is called {\em long-zigzag} 
	if there exists $\eps>0$ such that for every $n\in\N_0$ and every interval  $[a,b]\subset I$ such that $f^n([a, b])=[f^n(a), f^n(b)]$ (or $f^n([a,b])=[f^n(b), f^n(a)]$), $f^n(x)\not\in\{f^n(a),f^n(b)\}$ for every $x\in (a,b)$, and $|f^n(a)-f^n(b)|<\eps$, it holds that $f^n$ maps $[a,b]$ homeomorphically onto $f^n([a,b])$.
\end{definition}

\begin{remark}
	Map $f\colon I\to I$ is called {\em long-branched} if there exists $\eps>0$ such that for every $n\in\N$ and every adjacent critical points $a<b$ of $f^n$ it holds that $|f^n([a,b])|>\eps$. Note that every long-branched map is also long-zigzag. For example, Minc's map (see Figure~\ref{fig:Minc}) is long-branched, since the lengths of branches of $f^n$ are bounded below by $1/3$ for every $n\in \N$. However, there are non-longbranched maps which are long-zigzag. For example, there is a dense $G_{\delta}$ set of slope values in $(\sqrt{2},2]$ of such maps in the tent map family (see \cite{BBD}). Tent maps are zigzag-free, see Lemma 2.1 from \cite{ABC1}, and therefore tent maps are long-zigzag. Nevertheless, tent maps can have arbitrary small branches. 
\end{remark}

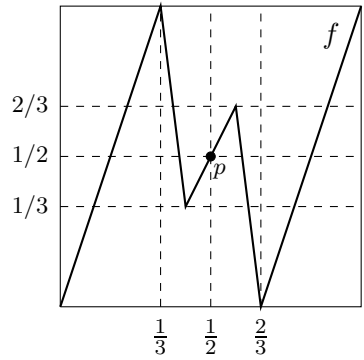
\begin{figure}
	\begin{tikzpicture}[scale=4]
	\draw [thick](0,0) -- (1/3, 1) -- (5/12, 1/3) -- (7/12,2/3) -- (2/3, 0) -- (1,1);
	\draw (0,0) -- (0,1) -- (1,1) -- (1,0) -- (0,0);
	\draw[dashed] (1/3, 0) -- (1/3, 1);
	\draw[dashed] (1/2, 0) -- (1/2, 1);
	\draw[dashed] (2/3, 0) -- (2/3, 1);
	\draw[dashed] (0, 1/3) -- (1, 1/3);
	\draw[dashed] (0, 1/2) -- (1, 1/2);
	\draw[dashed] (0, 2/3) -- (1, 2/3);
	
	\node at (1/3,-0.1) {\small $\frac{1}{3}$};
	\node at (1/2,-0.1) {\small $\frac{1}{2}$};
	\node at (2/3,-0.1) {\small $\frac{2}{3}$};
	
	\node at (-0.1,1/3) {\scriptsize $1/3$};
	\node at (-0.1,1/2) {\scriptsize $1/2$};
	\node at (-0.1,2/3) {\scriptsize $2/3$};
	
	\draw[solid, fill] (1/2,0.5) circle (0.015);
	\node at (0.9,0.9) {\small $f$};
	\node at (0.53,0.45) {\scriptsize $p$};
	\end{tikzpicture}
	\caption{Example by Minc (2001) suggested as a candidate for a counterexample to the Nadler-Quinn problem \cite{Nadler}. The map is long-branched and thus also long-zigzag. However, it is not zigzag-free.}
	\label{fig:Minc}
\end{figure}

In the following theorem we show that the set of long-zigzag maps is large. As a corollary we obtain that every locally eventually onto map $f\colon I\to I$ with a finite set of critical points is long-zigzag.

\begin{theorem}\label{thm:lzzbcond}
	Let $f\colon I\to I$ be a map with finite critical set $C=\{0=c_1<c_2,\ldots, c_N=1\}$, such that for every $i\in\{1,\ldots, N-1\}$ one of the following holds:
	\begin{itemize}
		\item[(a)] there exists $k(i)\in\N$ such that $f^{k(i)}((c_i,c_{i+1}))\cap C\neq\emptyset$, or
		\item[(b)] $f^k|_{[c_i,c_{i+1}]}$ is a homeomorphism onto its image for every $k\in\N$, and there is $\delta(i)>0$ such that $|f^k([c_i,c_{i+1}])|\geq\delta(i)$ for all $k\in\N$.
	\end{itemize} 
Then $f$ is a long-zigzag map.
\end{theorem}
\begin{proof}
	Let $\Sigma=\{i\in\{1, \ldots, N-1\}: i \textrm{ satisfies property (a)}\}$, and $\Delta=\{i\in\{1, \ldots, N-1\}: i \textrm{ satisfies property (b)}\}$. Then $\Sigma\cap \Delta=\emptyset$.
	
	For $i\in \Sigma$ let $k(i)\in\N$ be the minimal natural number such that there is a critical point of $f$ in $f^{k(i)}((c_i,c_{i+1}))$.
	
	Define 
	$$\eps:=\min\left\{\left\{\min_{1\leq j\leq k(i)}|f^j([c_i,c_{i+1}])|:i\in \Sigma\right\}\cup\{\delta(i): i\in \Delta\}\right\}$$
	
	and observe that $\eps>0$.
	We will show that for every $n\in\N$ and every $[a,b]\subset I$ such that $f^n|_{[a,b]}$ is not monotone, if $f^n([a,b])=[f^n(a),f^n(b)]$ (or $[f^n(b),f^n(a)]$) and  $f^n(x)\not\in\{f(a),f(b)\}$ for all $x\in(a,b)$, then it holds that $|f^n(a)-f^n(b)|\geq \eps$.

	Fix $n\in\N$ and $[a,b]\subset I$. By Observation~\ref{obs:compzz}, either $f|_{f^{n-1}([a,b])}$ is monotone, or $f$ contains at least two critical points in $f^{n-1}([a,b])$. In the later case $|f^n(a)-f^n(b)|\geq \min\{|f(c_i)-f(c_{i+1})|: i\in\{1,\ldots, N-1\}\}\geq \eps.$
	
	Assume $f|_{f^{n-1}([a,b])}$ is monotone. Again, either $f|_{f^{n-2}([a,b])}$ is monotone, or $f$ contains at least two critical points in $f^{n-2}([a,b])$. In the later case $|f^{n-1}(a)-f^{n-1}(b)|\geq \min\{|f(c_i)-f(c_{i+1})|: i\in\{1,\ldots, N-1\}\}.$ Since $f|_{f^{n-1}([a,b])}$ is monotone, it follows that $f^2|_{[c_i,c_{i+1}]}$ is monotone, and we conclude that $|f^n(a)-f^n(b)|\geq |f^2([c_i,c_{i+1}])|\geq \eps$.
	
	We continue inductively and conclude that $|f^n(a)-f^n(b)|\geq\eps$.
\end{proof}

\begin{definition}
	We say that a map $f\colon I\to I$ is {\em locally eventually onto (leo)} if for every $[a, b]\subset I$ there exists $n\in\N$ such that $f^n([a,b])=I$.
\end{definition}

\begin{corollary}\label{thm:leo+fpwl=lzzb}
	Every piecewise monotone leo map $f\colon I\to I$ is long-zigzag.
\end{corollary}
\begin{proof}
	Since for every $[c_i,c_{i+1}]$ there is $k\in\N$ such that $f^k([c_i, c_{i+1}])=I$, the proof follows directly from the previous theorem.
\end{proof}

\begin{remark}
	It is obviously not enough to assume only piecewise monotonicity of $f$ in Corollary~\ref{thm:leo+fpwl=lzzb}, see \eg map on Figure~\ref{fig:doublespiral1} and its iterations. 
\end{remark}

\begin{corollary}\label{cor:Raines}
	If $f\colon I\to I$ is piecewise monotone map with the property that for every interval $J\subset I$, and every connected component $A$ of $f^{-1}(J)$ it holds that $|A|\leq |J|$, then $f$ is long-zigzag.
\end{corollary}
\begin{proof}
	Let $C=\{0=c_1< c_2 < \ldots < c_N=1\}$ be the critical set of $f$. Note that for every $i\in\{1,\ldots, N-1\}$ for which $f^k|_{[c_i,c_{i+1}]}$ is one-to-one, it holds that $|f^k([c_i,c_{i+1}])|\geq |c_{i+1}-c_i|$. We define $\delta_i:=|c_{i+1}-c_i|$ and apply Theorem~\ref{thm:lzzbcond}. 
\end{proof}

\section{Types of endpoints}\label{sec:endpoints}

In this section we study different notions of an endpoint in interval inverse limits and how they emerge. We study $\underleftarrow{\lim}\{I, f_i\}$, where $f_i\colon I\to I$ is a continuous surjection for every $i\in\N$.

\begin{definition} 
	Let $K$ be a chainable continuum.
	We say that $x\in K$ is an {\em endpoint} if for every subcontinua $A, B\subset K$ such that $x\in A\cap B$, it holds that $A\subset B$ or $B\subset A$. We say that $x\in K$ is an {\em $L$-endpoint} (definition is due to Lelek \cite{Le}) if it is an endpoint of every arc in $K$ which contains it. 
\end{definition}

Note that if $x$ is in no arc of $K$, $x$ is automatically an $L$-endpoint.
Lelek's definition of an endpoint is usually used in the case when $K$ is a dendroid, and the first definition is used when $K$ is chainable. However, we will see that $L$-endpoints are of significant importance also in chainable continua. Note that every endpoint is also $L$-endpoint, but the converse is not true. For instance, every vertex of degree one in a simple triod is an $L$-endpoint but not an endpoint. Counterexamples can also be found among chainable continua, for example if we attach an arc to any point of the pseudo-arc; then the point of the intersection is an $L$-endpoint but not an endpoint.


In this section we partially answer the following problem, and use the obtained results in the subsequent sections.
\begin{problem}
 For which chainable continua $\underleftarrow{\lim}\{I, f_i\}$ are the two definitions of endpoints equivalent?
\end{problem}

Bruin shows in \cite{Br1} that the definitions are equivalent for inverse limits on intervals generated by a single {\em unimodal} bonding map. There he introduces a notion of a basic arc and then proves that in unimodal inverse limits a point $x$ is an endpoint if and only if it is an endpoint of its basic arc and is not ``glued" to any other basic arc in $x$. 
We generalize the notion of a basic arc here and then reformulate Bruin's result. 

\begin{definition}
	Let $X=\underleftarrow{\lim}\{I, f_i\}$ and $x=(x_0, x_1, \ldots)\in X$. For $i\in\N_0$ we define 
	the {\em $i$-basic arc} $A_i(x)$ as the maximal arc in $X$ such that $x\in A_i(x)$ and $\pi_i|_{A_i(x)}$ is one-to-one (can be degenerate).
\end{definition}

Observe that if $J_{i+k}\subset I$ are closed intervals such that $x_{i+k}\in J_{i+k}$ and $f_{i}\circ\ldots\circ f_{i+k-1}|_{J_{i+k}}$ is one-to-one and onto $J_i$ for every $k>0$, then $\{x\in X: x_{i+k}\in J_{i+k}, k> 0\}\subset A_i(x)$.

\begin{proposition}[Proposition 2 in \cite{Br1}]\label{prop:Bruin}
	Let $T_s$ be a tent map with slope $s\in (1, 2]$. Then $x\in \underleftarrow{\lim}\{I, T_s\}$ is an endpoint of $\underleftarrow{\lim}\{I, T_s\}$ if and only if it is an endpoint of $A_i(x)$ for every $i\in\N_0$ (see Figure~\ref{fig:uils}).
\end{proposition}

\begin{figure}
	\begin{tikzpicture}[scale=4]
	\draw (0,1.6)--(1,1.6);
	\draw (0,1.5)--(1,1.5);
	\draw (0,1.58)--(1,1.58);
	\draw (0,1.52)--(1,1.52);
	\draw (0,1.38)--(1,1.38);
	\draw (0,1.3)--(1,1.3);
	\draw (0,1.365)--(1,1.365);
	\draw (0,1.315)--(1,1.315);
	\draw (0,1.16)--(1,1.16);
	\draw (0,1.1)--(1,1.1);
	\draw (0,1.15)--(1,1.15);
	\draw (0,1.11)--(1,1.11);
	\draw[domain=90:270] plot ({0.05*cos(\x)}, {1.55+0.05*sin(\x)});
	\draw[domain=90:270] plot ({0.03*cos(\x)}, {1.55+0.03*sin(\x)});
	\draw[domain=90:270] plot ({0.04*cos(\x)}, {1.34+0.04*sin(\x)});
	\draw[domain=90:270] plot ({0.025*cos(\x)}, {1.34+0.025*sin(\x)});
	\draw[domain=90:270] plot ({0.03*cos(\x)}, {1.13+0.03*sin(\x)});
	\draw[domain=90:270] plot ({0.02*cos(\x)}, {1.13+0.02*sin(\x)});
	\draw[thick] (0, 1)--(1,1);
	\node[circle,fill, inner sep=1] at (0,1){};
	\node at (-0.07,0.98) {$x$};
	\node at (0.5,0.85) {\small flat endpoint ($\E_F$)};
	\end{tikzpicture}
	\begin{tikzpicture}[scale=1]
	\draw[dashed] (1,1) circle (1);
	\node[circle,fill, inner sep=1] at (1,1){};
	\node at (1,0.8) {\small $x$};
	\node at (0.5,1) {?};
	\node at (1.5,1) {?};
	\node at (1,1.5) {?};
	\node at (1,0.3) {?};
	\node at (0.9,-0.5) {\small nasty endpoint ($\E_N$)};
	\end{tikzpicture}
	\vspace{10pt}
	\begin{tikzpicture}[scale=5]
	\draw (0.1,0.75)--(0.8,0.75);
	\draw (0.3,0.62)--(0.8,0.62);
	\draw (0.3,0.56)--(0.65,0.56);
	\draw (0.4,0.53)--(0.65,0.53);
	\draw (0.4,0.51)--(0.55,0.51);
	\draw[domain=270:450] plot ({0.8+0.065*cos(\x)}, {0.685+0.065*sin(\x)});
	\draw[domain=90:270] plot ({0.3+0.03*cos(\x)}, {0.59+0.03*sin(\x)});
	\draw[domain=270:450] plot ({0.65+0.015*cos(\x)}, {0.545+0.015*sin(\x)});
	\draw[domain=90:270] plot ({0.4+0.01*cos(\x)}, {0.52+0.01*sin(\x)});
	\node[circle,fill, inner sep=1] at (0.5,0.47){};
	\node at (0.5,0.42) {$x$};
	\node at (0.5,0.3) {\small spiral endpoint ($\E_S$)};
	\end{tikzpicture}
	\caption{Different types of endpoints in unimodal inverse limits. For further properties of endpoints in the unimodal setting see \cite{AABC}.}
	\label{fig:uils}
\end{figure}
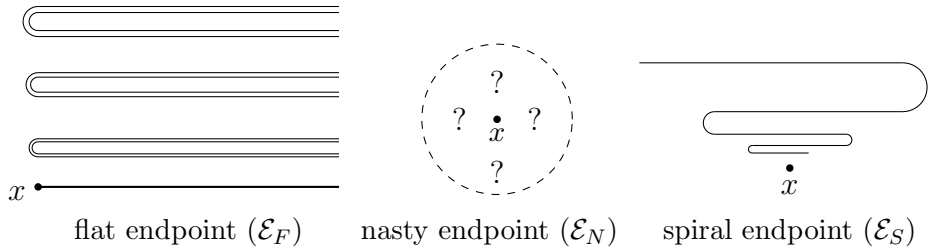

Proposition~\ref{prop:Bruin} motivates the third definition of an endpoint in $X=\underleftarrow{\lim}\{I, f_i\}$. 

\begin{definition}
	We say that $x\in X=\underleftarrow{\lim}\{I, f_i\}$ is a $B$-endpoint if it is an endpoint of $A_i(x)$ for every $i\in\N_0$.
\end{definition}

Note that the definition above depends on the inverse limit representation of $X$. For example, an arc can be represented as an inverse limit on $I$ generated by the identity, or as a double spiral generated by a bonding map in Figure~\ref{fig:doublespiral1}. In the first representation the whole inverse limit space is a single basic arc and a point is an endpoint if and only if it is a $B$-endpoint. In the latter representation, for spiral point $x=(1/2, 1/2, \ldots)$ every $A_i(x)$ is degenerate and thus $x$ is a $B$-endpoint. Obviously, $x$ is not an endpoint.

\begin{figure}
	\centering
	
	\begin{tikzpicture}[scale=5]
	\draw (0.1,0.75)--(0.8,0.75);
	\draw (0.3,0.62)--(0.8,0.62);
	\draw (0.3,0.56)--(0.65,0.56);
	\draw (0.4,0.53)--(0.65,0.53);
	\draw (0.4,0.51)--(0.55,0.51);
	\draw[domain=270:450] plot ({0.8+0.065*cos(\x)}, {0.685+0.065*sin(\x)});
	\draw[domain=90:270] plot ({0.3+0.03*cos(\x)}, {0.59+0.03*sin(\x)});
	\draw[domain=270:450] plot ({0.65+0.015*cos(\x)}, {0.545+0.015*sin(\x)});
	\draw[domain=90:270] plot ({0.4+0.01*cos(\x)}, {0.52+0.01*sin(\x)});
	\node[circle,fill, inner sep=1] at (0.5,0.47){};
	\begin{scope}[yscale=-1,xscale=-1,yshift=-0.94cm,xshift=-1cm]
	\draw (0.1,0.75)--(0.8,0.75);
	\draw (0.3,0.62)--(0.8,0.62);
	\draw (0.3,0.56)--(0.65,0.56);
	\draw (0.4,0.53)--(0.65,0.53);
	\draw (0.4,0.51)--(0.55,0.51);
	\draw[domain=270:450] plot ({0.8+0.065*cos(\x)}, {0.685+0.065*sin(\x)});
	\draw[domain=90:270] plot ({0.3+0.03*cos(\x)}, {0.59+0.03*sin(\x)});
	\draw[domain=270:450] plot ({0.65+0.015*cos(\x)}, {0.545+0.015*sin(\x)});
	\draw[domain=90:270] plot ({0.4+0.01*cos(\x)}, {0.52+0.01*sin(\x)});
	\end{scope}
	\end{tikzpicture}
	\hspace{20pt}
	\begin{tikzpicture}[scale=3]
	\draw (0,0)--(0,1)--(1,1)--(1,0)--(0,0);
	\draw[thick] (0,0)--(1/3,5/9)--(2/3,4/9)--(1,1);
	\draw[dashed] (0,5/9)--(5/9,5/9)--(5/9,0);
	\draw[dashed] (1,4/9)--(4/9,4/9)--(4/9,1);
	\node[circle,fill, inner sep=1] at (0.5,0.5){};
	\draw[dashed] (0,0)--(1,1);
	\end{tikzpicture}

	\caption{Representation of an arc in which point in the interior is a $B$-endpoint.}
		\label{fig:doublespiral1}
\end{figure}
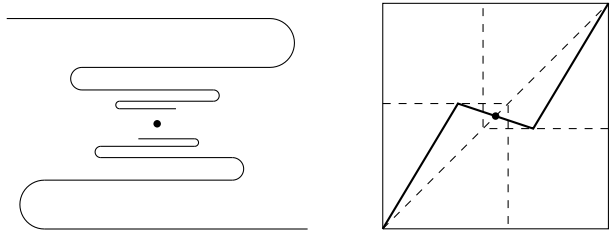

It follows directly from the definitions that every $L$-endpoint is a $B$-endpoint. For further examples of $B$-endpoints which are not $L$-endpoints see Figure~\ref{fig:one} and Figure~\ref{fig:two}. 

In general we have 
$$\{\text{endpoints}\}\subset \{L\text{-endpoints}\} \subset \{B\text{-endpoints}\}.$$ Thus we pose the following natural problem.  

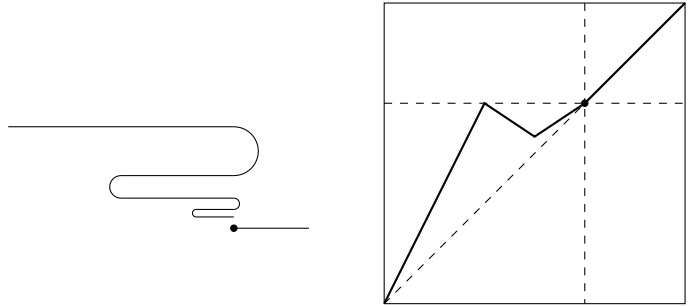
\begin{figure}
	\centering
	\begin{tikzpicture}[scale=5]
	\draw (0.2,0.75+0.2)--(0.8,0.75+0.2);
	\draw (0.5,0.62+0.2)--(0.8,0.62+0.2);
	\draw (0.5,0.56+0.2)--(0.8,0.56+0.2);
	\draw (0.7,0.53+0.2)--(0.8,0.53+0.2);
	\draw (0.7,0.51+0.2)--(0.8,0.51+0.2);
	\draw[domain=270:450] plot ({0.8+0.065*cos(\x)}, {0.685+0.2+0.065*sin(\x)});
	\draw[domain=90:270] plot ({0.5+0.03*cos(\x)}, {0.59+0.2+0.03*sin(\x)});
	\draw[domain=270:450] plot ({0.8+0.015*cos(\x)}, {0.545+0.2+0.015*sin(\x)});
	\draw[domain=90:270] plot ({0.7+0.01*cos(\x)}, {0.52+0.2+0.01*sin(\x)});
	\draw (0.8, 0.48+0.2)--(1,0.48+0.2);
	\node[circle,fill, inner sep=1] at (0.8,0.48+0.2){};
	\node[circle, inner sep=0] at (0.8,0.48){};
	\end{tikzpicture}
	\hspace{20pt}
	\begin{tikzpicture}[scale=4]
	\draw (0,0)--(0,1)--(1,1)--(1,0)--(0,0);
	\draw[dashed] (0,2/3)--(2/3,2/3)--(2/3,0);
	\draw[dashed] (2/3,1)--(2/3,2/3)--(1,2/3);
	\draw[dashed] (0,0)--(1,1);
	\draw[thick] (0,0)--(1/3,2/3)--(1/2,5/9)--(2/3,2/3)--(1,1);
	\node[circle,fill, inner sep=1] at (2/3,2/3){};
	\end{tikzpicture}
	\caption{$B$-endpoint which is not an $L$-endpoint.}
		\label{fig:one}
\end{figure}

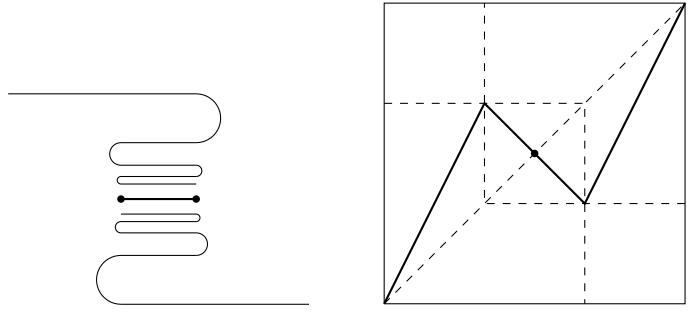
\begin{figure}
	\centering
	
	\begin{tikzpicture}[scale=5]
	\draw (0.1,0.75)--(0.6,0.75);
	\draw (0.4,0.62)--(0.6,0.62);
	\draw (0.4,0.56)--(0.6,0.56);
	\draw (0.4,0.53)--(0.6,0.53);
	\draw (0.4,0.51)--(0.6,0.51);
	\draw[domain=270:450] plot ({0.6+0.065*cos(\x)}, {0.685+0.065*sin(\x)});
	\draw[domain=90:270] plot ({0.4+0.03*cos(\x)}, {0.59+0.03*sin(\x)});
	\draw[domain=270:450] plot ({0.6+0.015*cos(\x)}, {0.545+0.015*sin(\x)});
	\draw[domain=90:270] plot ({0.4+0.01*cos(\x)}, {0.52+0.01*sin(\x)});
	\draw[thick] (0.4,0.47)--(0.6,0.47);
	\node[circle,fill, inner sep=1] at (0.4,0.47){};
	\node[circle,fill, inner sep=1] at (0.6,0.47){};
	\begin{scope}[yscale=-1,xscale=-1,yshift=-0.94cm,xshift=-1cm]
	\draw (0.1,0.75)--(0.6,0.75);
	\draw (0.4,0.62)--(0.6,0.62);
	\draw (0.4,0.56)--(0.6,0.56);
	\draw (0.4,0.53)--(0.6,0.53);
	\draw (0.4,0.51)--(0.6,0.51);
	\draw[domain=270:450] plot ({0.6+0.065*cos(\x)}, {0.685+0.065*sin(\x)});
	\draw[domain=90:270] plot ({0.4+0.03*cos(\x)}, {0.59+0.03*sin(\x)});
	\draw[domain=270:450] plot ({0.6+0.015*cos(\x)}, {0.545+0.015*sin(\x)});
	\draw[domain=90:270] plot ({0.4+0.01*cos(\x)}, {0.52+0.01*sin(\x)});
	\end{scope}
	\end{tikzpicture}
	\hspace{20pt}
	\vspace{10pt}
	\begin{tikzpicture}[scale=4]
	\draw (0,0)--(0,1)--(1,1)--(1,0)--(0,0);
	\draw[thick] (0,0)--(1/3,2/3)--(2/3,1/3)--(1,1);
	\draw[dashed] (0,2/3)--(2/3,2/3)--(2/3,0);
	\draw[dashed] (1,1/3)--(1/3,1/3)--(1/3,1);
	\node[circle,fill, inner sep=1] at (0.5,0.5){};
	\draw[dashed] (0,0)--(1,1);
	\end{tikzpicture}
	\caption{Example of a $B$-endpoint which is not an endpoint, where $X$ is not an arc.}
	\label{fig:two}
\end{figure}

\begin{problem}
	For which $\underleftarrow{\lim}\{I,f_i\}$ is every $B$-endpoint an endpoint?
\end{problem}

In this section we show that every B-endpoint is an endpoint in the case when all $f_i$
are zigzag-free or long-zigzag and leo. Later we study chainable continua for which all inhomogeneities are endpoints. It will be important to restrict to the cases in which endpoints are B-endpoints, see the proof of Theorem~\ref{thm:FisE}.


\begin{lemma}\label{lem:endpts}
	Assume that every map $f_i$ from $\underleftarrow{\lim}\{I, f_i\}$ is zigzag-free. Then $x\in \underleftarrow{\lim}\{I, f_i\}$ is an endpoint if and only if it is a $B$-endpoint. Specially, all three definitions of endpoints are equivalent.
\end{lemma}
\begin{proof}
Assume $x$ is not an endpoint, so there are subcontinua $A, B\subset \underleftarrow{\lim}\{I, f_i\}$ such that $x\in A\cap B$ and $A\setminus B, B\setminus A\neq\emptyset$.
Let $A_i=\pi_i(A), B_i=\pi_i(B)$, $i\in\N_0$ be coordinate projections. They are all intervals, $x_i\in A_i\cap B_i$ for every $i$, and there exists $N\in\N$ such that $A_{i}\setminus B_{i}, B_{i}\setminus A_{i}\neq\emptyset$, for all $i>N$.
Since $f_{i+1}$ does not contain a zigzag, there exists an interval $(l_{i+1}, r_{i+1})\ni x_{{i+1}}$ such that $f_{i+1}|_{[l_{i+1}, r_{i+1}]}\colon [l_{i+1}, r_{i+1}]\to A_{i}\cup B_{i}$ is one-to-one and surjective, see Figure~\ref{fig:nzz}.
So we can find an arc 
$A:= [a_N,b_N] \overset{f_{N+1}}{\longleftarrow} [a_{N+1},b_{N+1}] \overset{f_{N+2}}{\longleftarrow}[a_{N+2},b_{N+2}] \overset{f_{N+3}}{\longleftarrow} [a_{N+4}, b_{N+4}]\overset{f_{N+4}}{\longleftarrow}\ldots$
such that $x\in \Int(A)\subset \Int(A_{N}(x))$, so $x$ is not a $B$-endpoint.
\end{proof}

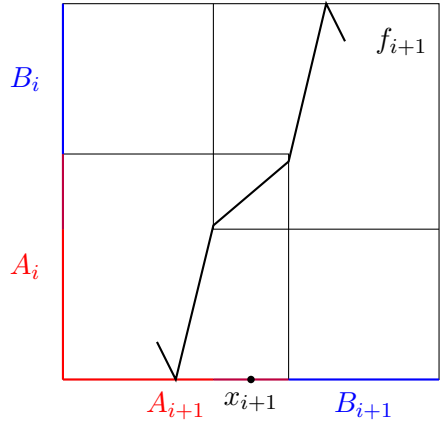
\begin{figure}
	\centering
	\begin{tikzpicture}[scale=5]
	\draw (0,0)--(1,0)--(1,1)--(0,1)--(0,0);
	\draw (0,0)--(0.6,0)--(0.6,0.6)--(0,0.6);
	\draw (1,0.4)--(0.4,0.4)--(0.4,1);
	\draw[thick,red] (0,0)--(0.4,0);
	\node[red] at (0.3,-0.07) {\small $A_{{i+1}}$};
	\draw[thick,purple] (0.4,0)--(0.6,0);
	\draw[thick,blue] (0.6,0)--(1,0);
	\node[blue] at (0.8,-0.07) {\small $B_{{i+1}}$};
	\draw[thick,red] (0,0)--(0,0.4);
	\node[red] at (-0.1,0.3) {\small $A_{{i}}$};
	\draw[thick,purple] (0,0.4)--(0,0.6);
	\draw[thick,blue] (0,0.6)--(0,1);
	\node[blue] at (-0.1,0.8) {\small $B_{{i}}$};
	\node at (0.9,0.9) {\small $f_{i+1}$};
	\node[circle,fill, inner sep=1] at (0.5,0){};
	\node at (0.5,-0.06) {\small $x_{{i+1}}$};
	\draw[thick] (0.25,0.1)--(0.3,0)--(0.4,0.41)--(0.6,0.58)--(0.7,1)--(0.75,0.9);
	\end{tikzpicture}
	\caption{Since $f_{i+1}$ is zigzag-free, for minimal interval $J\subset I$ such that $f_{i+1}(J)=A_i\cup B_i$ it holds that $f_{i+1}|_J$ is one-to-one.}
		\label{fig:nzz}
\end{figure}

The previous theorem can be generalized to the class of long-zigzag maps which are locally eventually onto (leo).

\begin{observation}\label{obs:leo}
	If $f$ is a leo map, and $A\subset \underleftarrow{\lim}\{I, f\}$ is a proper subcontinuum, then $|\pi_i(A)|\to 0$ as $i\to\infty$ (see \eg \cite[Lemma 2]{BrBr}). 
\end{observation}

\begin{theorem}\label{thm:llzzb}
	If $f\colon I\to I$ is long-zigzag leo map, then every $B$-endpoint in $\underleftarrow{\lim}\{I, f\}$ is an endpoint.
	In particular, the sets of $B$-endpoints, $L$-endpoints, and endpoints in $\underleftarrow{\lim}\{I,f\}$ are the same. 
\end{theorem}
\begin{proof}
	The proof follows as the proof of Lemma~\ref{lem:endpts}. We only need to find $N\in\N$ such that $|A_i\cup B_i|<\eps$ for all $i>N$. However, this is always possible by Observation~\ref{obs:leo}, since $f$ is leo.
\end{proof}

\begin{remark}
	There exist long-zigzag maps (not leo) for which some $B$-endpoints are not endpoints, see \eg Figure~\ref{fig:two}. However, every such $B$-endpoint must be an endpoint of its non-degenerate $0$-basic arc, as we show in the next proposition. We define double spiral points first and then show that in the inverse limits of long-zigzag maps they do not exist.  
\end{remark}

\begin{definition}\label{def:doublespiral}
 $B$-endpoint $x\in \underleftarrow{\lim}\{I, f_i\}$ which is not an $L$-endpoint is called a {\em double spiral point}. 
\end{definition}

\begin{proposition}\label{prop:doublespiral}
	If $f\colon I\to I$ is long-zigzag, then $\underleftarrow{\lim}\{I, f\}$ does not contain double spiral points. 
\end{proposition}
\begin{proof}
	Denote by $\eps>0$ the constant from the definition of a long-zigzag map. Assume that $x=(x_0, x_1, \ldots)\in \underleftarrow{\lim}\{I, f\}$ is a $B$-endpoint which is contained in the interior of a non-degenerate arc. That means that there exist non-degenerate arcs $A, B\subset \underleftarrow{\lim}\{I, f\}$ such that $\{x\}=A\cap B$. Denote by $A_i=\pi_i(A)$ and $B_i=\pi_i(B)$ for $i\geq 0$ and note that every $A_i, B_i$ are intervals in $I$, $x_i\in A_i\cap B_i$, and $f|_{A_i}, f|_{B_i}$ are surjective. Furthermore, since $A\cup B$ is an arc, we can take without loss of generality $|A_0\cup B_0|<\eps$. 
	
	Since $f$ is long-zigzag, for every $i\in\N$ we can find an interval $x_i\in U_i\subset A_i\cap B_i$ which is mapped homeomorphically onto $A_0\cap B_0$. 
	As in the proof of Lemma~\ref{lem:endpts} we conclude that $A_0(x)$ is non-degenerate, and $x\in \Int A_0(x)$, which is a contradiction.
\end{proof}

Note that for double spirals in Figure~\ref{fig:doublespiral1} and \ref{fig:one}, there exists a homeomorphic continuum with a different inverse limit representation $\underleftarrow{\lim}\{I,f\}$ for which there are no double spirals. Thus we ask the following question:

\begin{question}\label{question1}
	Given $X=\underleftarrow{\lim}\{I, f\}$, is it always possible to find continuous maps $g_i\colon I\to I$ such that $X$ is homeomorphic to $\underleftarrow{\lim}\{I, g_i\}$ and such that every $B$-endpoint in $\underleftarrow{\lim}\{I, g_i\}$ is an $L$-endpoint (that is, can we find a representation of $X$ in which there are no double spirals)? Can we do this requiring $g_i=g$ for every $i\in \N$?
\end{question}

\begin{example}\label{ex:leo+doublespiral}
	There exists a leo interval map $f$ (with infinitely many critical points) so that $\underleftarrow{\lim}(I,f)$ has a double spiral, see Figure~\ref{fig:doublespiralleo}.
	Namely, the intervals $[a_i',b_i']\subset I$ are such that $f([a_{i+1}',b_{i+1}'])\subset[a_i',b_i']$ for every $i\in\N_0$, so $A:=[a'_0,b'_0] \overset{f}{\longleftarrow} [a'_1,b'_1] \overset{f}{\longleftarrow}[a'_{2},b'_{2}]\overset{f}{\longleftarrow}\ldots$ is a well-defined subcontinuum of $\underleftarrow{\lim}\{I,f\}$. The map $f$ is constructed such that $A$ is a double spiral obtained in the similar way as in Figure~\ref{fig:doublespiral1}. There are intervals $[a_i,b_i]\subset [a_i',b_i']$ for every $i\in\N_0$ for which $f|_{[a_i,b_i]}$ is one-to-one for every $i\in\N_0$,  $f{[a_{i+1},b_{i+1}]}\subset [a_i,b_i]$, and $\diam f^{i-1}([a_i,b_i])\to 0$ as $i\to \infty$. It is not difficult to see that $x=(x_0,x_1,x_2,\ldots)\in A$ for which $x_i\in[a_i,b_i]$ for every $i\in\N_0$ is a double spiral point.
	Moreover, we set $f([a'_0,b'_0])=I$. Then for every interval $J\subset I$ such that $[a_n',b_n']\subset J$ it holds that $f^n(J)=I$. We can set $f|_{[a_i',b_i']}$ to be of slope $3$ or more, and $f|_{[b_i',a_{i-1}']}$ to be of slope greater than $1$ for every $i\geq 1$. Therefore, if $J$ does not contain any $[a_n',b_n']$, then the diameter of $f^k(J)$ increases as $k$ increases. We only need to find $k\in\N$ for which $f^k(J)$ contains some $[a_n',b_n']$ and conclude that $f^{k+n}(J)=I$.
\end{example}

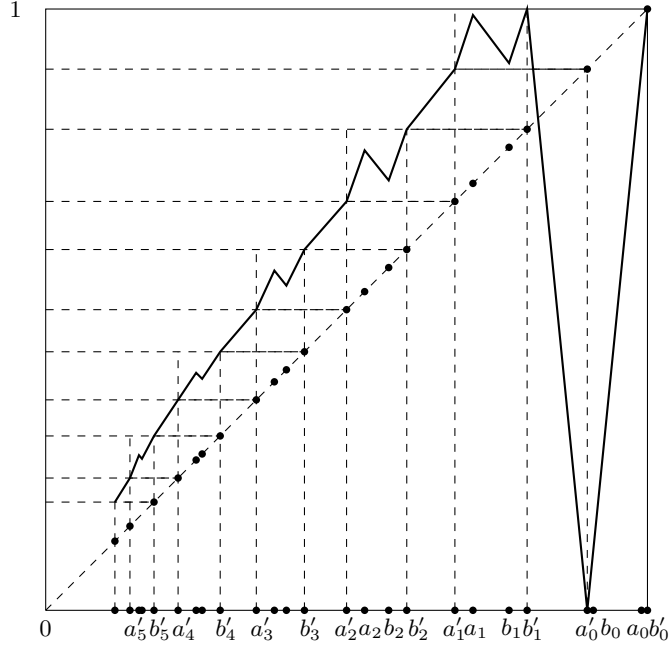
\begin{figure}
	\begin{tikzpicture}[scale=8]
	\draw[thick] (1,1)--(9/10,0)--(8/10,1)--(0.77,0.91)--(0.71,0.99)--(0.68,0.9)--(0.6,0.8)--(0.57,0.715)--(0.53,0.765)--(0.5,0.68)--(0.43,0.6)--(0.4,0.54)--(0.38,0.565)--(0.35,0.5)--(0.29,0.43)--(0.26,0.385)--(0.25,0.395)--(0.22,0.35)--(0.18,0.29)--(0.16,0.252)--(0.155,0.258)--(0.14,0.22)--(0.115,0.18);
	\draw[dashed] (0.9,0.9)--(0.8,0.9);
	\draw[dashed](0.14,0.14)--(0.14,0.22)--
	(0.22,0.22)--(0.22,0.35)--(0.35,0.35)--(0.35,0.5)--(0.5,0.5)--(0.5,0.68)--(0.68,0.68)--(0.68,0.9)--(0.9,0.9)--(0.9,0);
	\draw[dashed](0.115,0.115)--
	(0.115,0.18)--(0.18,0.18)--(0.18,0.29)--(0.29,0.29)--(0.29,0.43)--(0.43,0.43)--(0.43,0.6)--(0.6,0.6)--(0.6,0.8)--(0.8,0.8)--(0.8,1);
	\draw[dashed] (0,0)--(1,1);
	\draw (0,0) -- (0,1) -- (1,1) -- (1,0) -- (0,0);
	\draw[dashed] (0.8,0)--(0.8,0.8); 
	\draw[dashed] (0.68,0)--(0.68,0.68); 
	\draw[dashed] (0.6,0)--(0.6,0.6); 
	\draw[dashed] (0.5,0)--(0.5,0.5); 
	\draw[dashed] (0.43,0)--(0.43,0.43); 
	\draw[dashed] (0.35,0)--(0.35,0.35); 
	\draw[dashed] (0.29,0)--(0.29,0.29); 
	\draw[dashed] (0.22,0)--(0.22,0.22); 
	\draw[dashed] (0.18,0)--(0.18,0.18); 
	\draw[dashed] (0.14,0)--(0.14,0.14); 
	\draw[dashed] (0.115,0)--(0.115,0.115);
	
	
	\draw[dashed] (0,0.9)--(0.9,0.9); 
	\draw[dashed] (0,0.8)--(0.8,0.8); 
	\draw[dashed] (0,0.68)--(0.68,0.68); 
	\draw[dashed] (0,0.6)--(0.6,0.6); 
	\draw[dashed] (0,0.5)--(0.5,0.5); 
	\draw[dashed] (0,0.43)--(0.43,0.43); 
	\draw[dashed] (0,0.35)--(0.35,0.35); 
	\draw[dashed] (0,0.29)--(0.29,0.29); 
	\draw[dashed] (0,0.22)--(0.22,0.22); 
	\draw[dashed] (0,0.18)--(0.18,0.18); 

	\draw[dashed] (0.68,0.9)--(0.68,1); 
	\draw[dashed] (0.5,0.68)--(0.5,0.8); 
	\draw[dashed] (0.35,0.5)--(0.35,0.6); 
	\draw[dashed] (0.22,0.35)--(0.22,0.43); 
	\draw[dashed] (0.14,0.22)--(0.14,0.29);
	 
	\node[circle,fill, inner sep=1] at (1,1){};
	\node[circle,fill, inner sep=1] at (0.9,0.9){};
	\node[circle,fill, inner sep=1] at (0.8,0.8){};
	\node[circle,fill, inner sep=1] at (0.77,0.77){};
	\node[circle,fill, inner sep=1] at (0.71,0.71){};
	\node[circle,fill, inner sep=1] at (0.68,0.68){};
	\node[circle,fill, inner sep=1] at (0.6,0.6){};
	\node[circle,fill, inner sep=1] at (0.57,0.57){};
	\node[circle,fill, inner sep=1] at (0.53,0.53){};
	\node[circle,fill, inner sep=1] at (0.5,0.5){};
	\node[circle,fill, inner sep=1] at (0.43,0.43){};
	\node[circle,fill, inner sep=1] at (0.4,0.4){};
	\node[circle,fill, inner sep=1] at (0.38,0.38){};
	\node[circle,fill, inner sep=1] at (0.35,0.35){};
	\node[circle,fill, inner sep=1] at (0.29,0.29){};
	\node[circle,fill, inner sep=1] at (0.26,0.26){};
	\node[circle,fill, inner sep=1] at (0.25,0.25){};
	\node[circle,fill, inner sep=1] at (0.22,0.22){};
	\node[circle,fill, inner sep=1] at (0.18,0.18){};
	\node[circle,fill, inner sep=1] at (0.14,0.14){};
	\node[circle,fill, inner sep=1] at (0.115,0.115){};
	\node[circle,fill, inner sep=1] at (1,0){};
	\node[circle,fill, inner sep=1] at (0.99,0){};
	\node[circle,fill, inner sep=1] at (0.91,0){};
	\node[circle,fill, inner sep=1] at (0.9,0){};
	\node[circle,fill, inner sep=1] at (0.8,0){};
	\node[circle,fill, inner sep=1] at (0.77,0){};
	\node[circle,fill, inner sep=1] at (0.71,0){};
	\node[circle,fill, inner sep=1] at (0.68,0){};
	\node[circle,fill, inner sep=1] at (0.6,0){};
	\node[circle,fill, inner sep=1] at (0.5,0){};
	\node[circle,fill, inner sep=1] at (0.43,0){};
	\node[circle,fill, inner sep=1] at (0.35,0){};
	\node[circle,fill, inner sep=1] at (0.29,0.){};
	\node[circle,fill, inner sep=1] at (0.22,0){};
	\node[circle,fill, inner sep=1] at (0.18,0){};
	\node[circle,fill, inner sep=1] at (0.14,0){};
	\node[circle,fill, inner sep=1] at (0.115,0){};
	\node[circle,fill, inner sep=1] at (0.57,0){};
	\node[circle,fill, inner sep=1] at (0.53,0){};
	\node[circle,fill, inner sep=1] at (0.4,0){};
	\node[circle,fill, inner sep=1] at (0.38,0.){};
	\node[circle,fill, inner sep=1] at (0.25,0){};
	\node[circle,fill, inner sep=1] at (0.26,0){};
	\node[circle,fill, inner sep=1] at (0.16,0){};
	\node[circle,fill, inner sep=1] at (0.155,0){};
	\node at (-0.05,1) {\scriptsize $1$};
	\node at (0,-0.03) {\scriptsize $0$};
	\node at (0.15,-0.03) {\scriptsize $a'_5$};
	\node at (0.19,-0.03) {\scriptsize $b'_5$};
	\node at (0.23,-0.03) {\scriptsize $a'_{4}$};
	\node at (0.3,-0.03) {\scriptsize $b'_4$};
	\node at (0.36,-0.03) {\scriptsize $a'_3$};
	\node at (0.44,-0.03) {\scriptsize $b'_{3}$};
	\node at (0.5,-0.03) {\scriptsize $a'_2$};
	\node at (0.54,-0.035) {\scriptsize $a_{2}$};
	\node at (0.58,-0.03) {\scriptsize $b_{2}$};
	\node at (0.62,-0.03) {\scriptsize $b'_{2}$};
	\node at (0.68,-0.03) {\scriptsize $a'_{1}$};
	\node at (0.715,-0.035) {\scriptsize $a_{1}$};
	\node at (0.775,-0.03) {\scriptsize $b_{1}$};
	\node at (0.81,-0.03) {\scriptsize $b'_{1}$};
	\node at (0.9,-0.03) {\scriptsize $a'_{0}$};
	\node at (0.94,-0.03) {\scriptsize $b_{0}$};
	\node at (0.985,-0.035) {\scriptsize $a_{0}$};
	\node at (1.02,-0.03) {\scriptsize $b'_{0}$};
	\end{tikzpicture}
	\caption{Example of an interval leo map $f$ for which the inverse limit has double spirals.}
	\label{fig:doublespiralleo}
\end{figure}

\begin{remark}
	The characterization of $L$-endpoints or double spirals in  $\underleftarrow{\lim}\{I,f_i\}$ is tightly connected with the characterization of $\underleftarrow{\lim}\{I,f_i\}$ which are arcs. Such a characterization is in general tedious and technical, but in 
	Appendix~\ref{appB} we give an easily accessible and applicable characterization of arcs  $\underleftarrow{\lim}\{I,f\}$ for piecewise monotone bonding maps $f$ which is easy to describe in terms of the dynamics of the bonding map.
\end{remark}

\section{Folding points}\label{sec:fp}

We say that a point $x\in\underleftarrow{\lim}\{I, f_i\}$ is a {\em folding point} if it does not have a neighbourhood homeomorphic to $S\times (0,1)$, where $S$ is a zero-dimensional set.

Raines proved in \cite[Corollary 4.8]{Raines} that, if $f\colon I\to I$ satisfies properties (i)+(ii) (see below), then $x\in\underleftarrow{\lim}\{I, f\}$ is a folding point if and only if $\pi_n(x)\in\omega(C)$ for every $n\in\N_0$.

\begin{enumerate}
	\item[(i)] $f$ is piecewise monotone,
	\item[(ii)] If $J\subset I$ is an interval, then $|A|\leq |J|$ for every component $A$ of $f^{-1}(J)$.
\end{enumerate}

By Corollary~\ref{cor:Raines}, every $f$ which satisfies (i)+(ii) is long-zigzag.
We will generalize Raines' result to the class of long-zigzag maps. That is indeed a larger class, since  long-zigzag maps can have infinitely many critical points, or not satisfy property (ii). The proof crucially uses the fact that for long-zigzag map $f$ it holds that $\underleftarrow{\lim}\{I, f\}$ contains no double spirals.

The characterization of folding points obtained here will be used later in the proof of Theorem~\ref{thm:FisE}, with the assumption that $f$ is either zigzag-free or long-zigzag and locally eventually onto. In this setting we know that then point in the inverse limit is an endpoint if and only if it is an $L$-endpoint by Theorem~\ref{thm:llzzb}. 

\begin{theorem}\label{thm:folding}
	Assume $f\colon I\to I$ is a long-zigzag map. Then $x\in\underleftarrow{\lim}\{I, f\}$ is a folding point if and only if $\pi_n(x)\in\omega(C)$ for every $n\in\N_0$.
\end{theorem}
\begin{proof}
	Denote by $\eps>0$ the constant from the definition of long-zigzag map, \ie for every $n\in\N_0$ and minimal $[a, b]\subset I$ such that $f^n([a, b])=[f^n(a), f^n(b)]$ (or $f^n([a, b])=[f^n(b), f^n(a)]$) and $|f^n(a)-f^n(b)|<\eps$ it holds that $f^n|_{[a,b]}$ is one-to-one. Denote by $x=(x_0, x_1, \ldots)\in \underleftarrow{\lim}\{I, f\}$.
	
	($\Rightarrow$) Assume first that there exists $\delta>0$ such that $(x_0-\delta, x_0+\delta)\cap\{f^n(C):n\in\N_0\}=\emptyset$ and study $0$-box $U:=\pi^{-1}_0((x_0-\delta, x_0+\delta))$. Let $A\subset U$ be a maximal connected set. Then we can write $A=\underleftarrow{\lim}\{A_i, f|_{A_i}\}$, where $A_i\subset I$ is a maximal connected set such that $f(A_i)=A_{i-1}$, for every $i\geq 1$, and $A_0=(x_0-\delta, x_0+\delta)$. Note that a priori $f|_{A_i}$ does not need to be a surjection onto $A_{i-1}$. However, note that $A_i$ does not contain a critical point $c\in C$, since otherwise $f^i(c)\in A_0$, which contradicts the assumption. We conclude that $\pi_0|_A$ is one-to-one onto $(x_0-\delta,x_0+\delta)$. Thus, $U=(x_0-\delta,x_0+\delta)\times S$, where $S=\pi_0^{-1}(x_0)\subset \underleftarrow{\lim}\{I, f\}$ is one-dimensional (otherwise $x_0\in f^n(C)$). By applying $\widehat{f}^n$ to $x$ we can argue analogously if $x_n\notin \omega(C)$. This covers one side of the proof. Note also that here we did not use the fact that $f$ is long-zigzag.
	
	 ($\Leftarrow$) Let $x_n\in\omega(C)$ for every $n\geq 0$ and assume by contradiction that $x$ is not a folding point. Recall from Proposition~\ref{prop:doublespiral} that there are no double spirals so the $0$-basic arc $A_0(x)$ is non-degenerate. Also, we can assume that $x_0$ is contained in the interior of $A_0(x)$ (otherwise we find the smallest $i\in \N$ so that $x_i\in \Int(A_i(x))$). Take $0<\delta<\eps$ such that $\pi_0(A_0(x))\supset (x_0-\delta, x_0+\delta)$ and let $J\subset A_0(x)$ be such that $\pi_0(J)=(x_0-\delta, x_0+\delta)$. Denote by $J_n=\pi_n(J)$ for $n\in\N_0$. We will study the open neighbourhoods of $J$ given by $U_n:=\pi_n^{-1}(J_n)$ for $n\in\N_0$. Note that, since $n$-boxes form a basis for the topology, and after shrinking $\delta$ if necessary, we can find $N\in\N$ such that $U_n$ is homeomorphic to a zero-dimensional set times an arc for every $n\geq N$.\\
	Fix $n\geq N$ and let us study the arcs of maximal length in $U_n$. Every such arc is of the form $B=\underleftarrow{\lim}\{B_i,f\}$, where $B_n=J_n$, and $B_{k+n}$ is maximal such that $f^k(B_{k+n})\subset J_n$. Note that $f|_{B_i}$ is a homeomorphism onto the image for every $1\leq i<n$. If $f|_{B_j}$ is also a homeomorphism for $j\geq n$, then endpoints of $B$ are projected via $\pi_0$ to $x_0-\delta$ and $x_0+\delta$. Since $x_0\in\omega(C)$, there exist $\underleftarrow{\lim}\{B_i,f\}$ for which there is $j\geq n$ such that $f|_{B_j}$ is not one-to-one. Let us take the smallest such $j$. Then, since $\delta<\eps$, $f|_{B_j}$ is not onto and $f$ maps endpoints of the arc $B_j$ to the same point. Moreover, since $f$ is surjective, for every $k\geq 1$ we can find $B_{j+n+k}$ such that $f|_{B_{j+n+k}}$ is onto $B_{j+n+k-1}$. Denote by $Q_n=\underleftarrow{\lim}\{B_i, f\}$; it must be an arc since $U_n$ is a zero-dimensional set of arcs. Moreover, by construction, the endpoints of $Q_n$ are projected via $\pi_n$ to the same point. Specifically, the endpoints are projected via $\pi_0$ to the same point (either $x_0+\delta$ or $x_0-\delta$).
	Note that an arc $Q_n$ exists for infinitely $n\geq N$. Since the neighbourhoods $U_n$ of $A_0(x)$ shrink as $n$ increases, there exists a sequence of arcs $(Q_n)_{n\geq N}$ in $U_0$ such that $\partial Q_n$ converges to an endpoint of $A_0(x)$ as $n\to\infty$. 
	However, this implies that such $U_0$ is not homeomorphic to a zero dimensional set of arcs, a contradiction.
\end{proof}

\begin{remark}
	In general, it can happen that $\pi_n(x)\in\omega(C)$ for every $n\geq 0$, but $x$ is locally homeomorphic 
	to a zero dimensional set of arcs, see for example Figure~\ref{fig:doublespiral1}. In that case $x$ is contained in a double spiral. In general it is very difficult to check if a double spiral point is a folding point or not, see \eg Example~\ref{ex:leo+doublespiral}. A positive answer to Question~\ref{question1} would give a nice way around that difficulty.
\end{remark}

\section{Retractability along $\omega(C)$ and endpoints}\label{sec:persistence}

We generalize the notion of {\em persistent recurrence} introduced by Blokh and Lyubich in \cite{BL}. Namely we define when an interval map is non-retractable and show that a long-zigzag leo map $f\colon I \to I$ is non-retractable along $\omega(C)$ if and only if the {\em set of endpoints} $\E$ equals the {\em set of folding points} $\F$ in $\underleftarrow{\lim}\{I,f\}$.

\begin{definition}\label{def:pullback} Let $f\colon I\to I$ be a map with critical set $C$. 
	Let $x=(x_0, x_{1}, \ldots)\in \underleftarrow{\lim}\{I, f\}$ and let $J\subset I$ be an interval. 
	The sequence $(J_n)_{n\in\N_0}$ of intervals is called a \emph{pull-back} of $J$ along 
	$x$ if $J=J_0$, $x_{k}\in J_k$ and $J_{k+1}$ is the largest interval such that $f(J_{k+1})\subset J_k$ 
	for all $k\in\N_0$. A pull-back is {\em monotone} if $C\cap \Int( J_n)=\emptyset$ 
	for every $n\in\N$. 
\end{definition}

\begin{definition}\label{def:persrec}
	Let $f\colon I\to I$ be a map with critical set $C$. We say that $f$ is {\em retractable along $\omega(C)$} if there exists a
	backward orbit $x=(x_0, x_1, x_2, \ldots)$ in $\omega(C)$ and an open interval $U\subset I$ such that $x_0\in U$ and such that $U$ has a  monotone pull-back along $x$. Otherwise,  $f$ is called {\em non-retractable along $\omega(C)$}. We will often only write {\em non-retractable}.
\end{definition}

\begin{remark}
	Note that the notion of non-retractable map does not require recurrence of critical points or of the critical set $C$. For example, the map in Figure~\ref{fig:twosided} has the property that $\omega(C)\cap C=\emptyset$, but $f$ is non-retractable.
\end{remark}

The motivation for studying when $\E=\F$ comes from the topology of inverse limits of infinitely renormalisable unimodal maps, where it is known that $\F=\E$, and $f|_{\omega(c)}$ is conjugate to an adding machine. Later, {\em strange adding machines} were discovered in non-infinitely renormalisable unimodal maps \cite{BKM}. That led to a hypothesis that $\F=\E$ might be related to embedded adding machines, or at least the property that  $f|_{\omega(c)}$ is one-to-one, see \cite{Al}. However, examples in \cite{AlBr} showed the hypothesis to be false. In \cite{AABC} it is proven that the crucial notion for $\F=\E$ is {\em persistent recurrence}. It turns out that $\F=\E$ if and only if the critical point $c$ of a unimodal map is persistently recurrent. We generalize this to inverse limits of leo long-zigzag  maps $f$ in the following theorem. The more general proof is even simpler than the proof of Theorem 4.13 in \cite{AABC}, once we understand the notions of folding points and endpoints well enough.

\begin{theorem}\label{thm:FisE}
	Let $f\colon I\to I$ be a long-zigzag leo map with the critical set $C$. Then for $\underleftarrow{\lim}\{I, f\}$ it holds that all folding points are endpoints if and only if $f$ is non-retractable along $\omega(C)$.
\end{theorem}
\begin{proof}
	Since $f$ is long-zigzag, we know that $x=(x_0, x_1, \ldots)\in \underleftarrow{\lim}\{I, f\}$ is a folding point if and only if $x_n\in\omega(C)$ for every $n\geq 0$ due to Theorem~\ref{thm:folding}. Also, since $f$ is additionally leo, we know that $x$ is an endpoint if and only if it is a $B$-endpoint (if and only if it is an $L$-endpoint) due to Theorem~\ref{thm:llzzb}. So, $x$ is not an endpoint if and only if there is $n\geq 0$ such that $(x_n, x_{n+1}, \ldots)$ is contained in the interior of its $0$-basic arc. 
	
	If $f$ is retractable, there exists a folding point $x=(x_0, x_{1}, \ldots)\in \underleftarrow{\lim}\{I, f\}$, an interval $J\subset I$ such that 
	$x_0\in \Int(J)$, and an infinite monotone pull-back $(J_n)_{n\in\N_0}$ of $J$ along $x$. 
	Therefore, $\underleftarrow{\lim}\{J_n, f|_{J_n}\}$ is an arc in $\underleftarrow{\lim}\{I, f\}$ and it contains $x$ in its interior, and
	thus $x$ is not an endpoint.
	
	For the other direction, let $f$ be non-retractable and assume that there is a folding point $x=(x_0, x_1, \ldots)\in \underleftarrow{\lim}\{I, f\}$ which is not an endpoint. Without loss of 
	generality we can assume that $x$ is contained in the interior of its $0$-basic arc. Otherwise, we use 
	$\hat f^{-j}(x)$ for some $j\in\N$ large enough. Let $A\subset A_0(x)$ be an open subset of the $0$-basic arc of $x$ such that $x\in \Int(A)$ and such that $A$ does not contain endpoints of $A_0(x)$. Then $(\pi_n(A))_{n\in\N_0}$ is a monotone pull-back of $\pi_0(A)$ along $x$, so $f$ is retractable, a contradiction. 
\end{proof}

Appendix~\ref{appA} is a natural continuation of this section, but since it is considerably more technical, we moved it to the end of the paper. It gives a characterization of interval maps that are non-retractable along $\omega(C)$, through the dynamical properties of the maps.

\section{Endpoints and critical recurrence}\label{sec:recurrence}

Barge and Martin \cite{BM} give a characterization of endpoints in $\underleftarrow{\lim}\{I, f\}$ in terms of dynamics of map $f$. They prove, among other things, that if $f$ has finitely many critical points, a dense orbit, and if $\omega(C)\cap C=\emptyset$, then there are no endpoints. Assumption of dense orbit is indeed required, \eg we can take $f$ which generates a two-sided spiral as in Figure~\ref{fig:twosided} for an example that confirms that. In this section we establish a partial converse of the Barge and Martin result.

First, we show with the following example that only the condition  $\omega(C)\cap C=\emptyset$ is not enough to obtain the converse of the statement above. Thus, we need to impose additional assumptions on $f$.

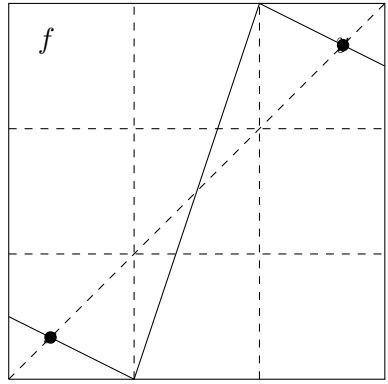
\begin{figure}
	\begin{tikzpicture}[scale=5]
	\draw (0,1/6) -- (1/3, 0) -- (2/3, 1) -- (1,5/6);
	\draw (0,0) -- (0,1) -- (1,1) -- (1,0) -- (0,0);
	\draw[dashed] (1/3, 0) -- (1/3, 1);
	\draw[dashed] (2/3, 0) -- (2/3, 1);
	\draw[dashed] (0, 1/3) -- (1, 1/3);
	\draw[dashed] (0, 2/3) -- (1, 2/3);
	
	\draw[dashed] (0, 0) -- (1, 1);
	
	\draw[solid, fill] (1/9,1/9) circle (0.015);
	\draw[solid, fill] (8/9,8/9) circle (0.015);
	\node at (0.1,0.9) {\small $f$};
	\node at (1/9,1/9) {\scriptsize $x$};
	\node at (8/9,8/9) {\scriptsize $y$};
	\end{tikzpicture}
	\caption{Example for which $\omega(C)\cap C=\emptyset$ but $\protect\underleftarrow{\lim}\{I, f\}$ is an arc and thus has endpoints. Actually, endpoints are $(x,x,x, \ldots)$ and $(y,y,y,\ldots)$, $x,y\neq 1/2$.}
	\label{fig:twosided}
\end{figure}

\begin{example}
	There exists an interval map $f$ so that $\underleftarrow{\lim}(I,f)$ has no endpoints but $\omega(C)\cap C\neq\emptyset$, see map $f$ in Figure~\ref{fig:counter}. The three intervals $I_1=[0,1/3], I_2=[1/3,2/3], I_3=[2/3,1]\subset I$ are $f$-invariant and $\underleftarrow{\lim}(I,f)=\underleftarrow{\lim}(I_1,f|_{I_1})\cup \underleftarrow{\lim}(I_2,f|_{I_2})\cup \underleftarrow{\lim}(I_3,f|_{I_3})$. Note that $\underleftarrow{\lim}(I_1,f|_{I_1})$ and $\underleftarrow{\lim}(I_3,f|_{I_3})$ are disjoint Knaster continua and $\underleftarrow{\lim}(I_2,f|_{I_2})$ are two $\sin(1/x)$-continua that share the same bar. Furthermore, $\underleftarrow{\lim}(I_1,f|_{I_1})\cap\underleftarrow{\lim}(I_2,f|_{I_2})=\{(1/3,1/3, \ldots)\}$ is the endpoint of the Knaster continuum $\underleftarrow{\lim}(I_1,f|_{I_1})$ and the endpoint of one of the rays in $\underleftarrow{\lim}(I_2,f|_{I_2})$. 
	A similar conclusion holds for $\underleftarrow{\lim}(I_2,f|_{I_2})\cap\underleftarrow{\lim}(I_3,f|_{I_3})=\{(2/3,2/3,\ldots)\}$. Therefore $\underleftarrow{\lim}(I,f)$ has no endpoints. However, $\omega(C)\cap C\neq\emptyset$ since the two critical points $1/3$ and $2/3$ are $2$-periodic.
	
\end{example}

\begin{figure}
	\begin{tikzpicture}[scale=5]
	\draw[thick] (0,1/3) -- (1/6, 0) -- (1/3, 1/3) -- (4/9,5/9)--(5/9,4/9)--(2/3,2/3)--(5/6,1)--(1,2/3);
	\draw (0,0) -- (0,1) -- (1,1) -- (1,0) -- (0,0);
	\draw[dashed] (1/3, 0) -- (1/3, 1);
	\draw[dashed] (2/3, 0) -- (2/3, 1);
	\draw[dashed] (0, 1/3) -- (1, 1/3);
	\draw[dashed] (0, 2/3) -- (1, 2/3);
	
	\draw[dashed] (0, 0) -- (1, 1);
	\draw[dashed] (4/9, 4/9) -- (4/9, 5/9)-- (5/9,5/9)--(5/9,4/9) -- (4/9,4/9);
	\draw[solid, fill] (4/9,5/9) circle (0.012);
	\draw[solid, fill] (5/9,4/9) circle (0.012);
	\node at (0.1,0.9) {\small $f$};
	\end{tikzpicture}
	\caption{Example of an interval map $f$ such that $\protect\underleftarrow{\lim}(I,f)$ has no endpoints but $\omega(C)\cap C\neq\emptyset$.}
	\label{fig:counter}
\end{figure}
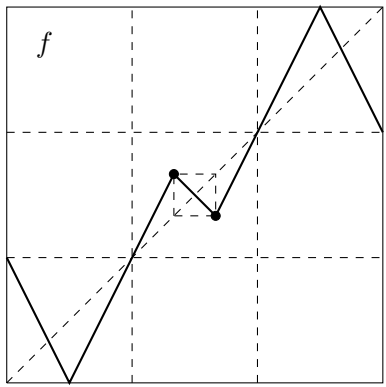

\begin{lemma}\label{lem:uncountably}
	Assume that $f$ is zigzag-free and there is $c\in C$ such that $c\in\omega(c)$. Then $\underleftarrow{\lim}\{I,f\}$ has an endpoint. If additionally $\orb(c)$ is infinite, then $\underleftarrow{\lim}\{I, f\}$ has uncountably many endpoints.
\end{lemma}
\begin{proof}
	
	Denote by $K_0\subset I$ an interval such that $f^{k_1}(c)\in\Int(K_0)$ for some $k_1\in\N$. If $f^k(c)\in\{0,1\}$ for every $k\in\N$, then $(0,0,\ldots)$ or $(0,1,0,1,\ldots)\in\underleftarrow{\lim}\{I,f\}$ and they are endpoints, see Theorem 2.9  in \cite{BM} for the detailed argument. We choose an arbitrary small interval $L_0$ such that $f^{k_1}(c)\in\Int (L_0)$, and an interval $K_{k_1}$ such that $c\in\Int(K_{k_1})$ and $f^{k_1}(K_{k_1})\subset L_0$. 
	
	Since $c\in\omega(c)$, we can continue inductively and for every $j\geq 2$ find $k_j\in\N$ such that $f^{k_j}(c)\in\Int(K_{k_{j-1}})$ (infinitely many if $\orb(c)$ is infinite). We choose an arbitrary small interval $L_{j-1}$ such that $f^{k_j}(c)\in\Int(L_{j-1})$, and an interval $K_{k_j}$ such that $c\in\Int(K_{k_j})$, $f^{k_j}(K_{k_j})\subset L_{j-1}$. 
	
	Since the intervals $L_j$ can be chosen arbitrarily small, and since $L_j$ is closed for every $j\in\N_0$, we can define $x_0=\cap_{i\geq 1} f^{k_1+k_2+\ldots +k_i}(L_i)\subset L_0$, and $x_{k_j}=\cap_{i\geq j+1} f^{k_{j+1}+k_{j+2}+\ldots +k_i}(L_i)\\
	\subset L_j$, for $j\geq 1$. We claim that $x=(x_0,x_1,\ldots, x_{k_1},x_{k_1+1}, \ldots)\in\underleftarrow{\lim}\{I,f\}$ is an endpoint. Note that since $L_j$ can be chosen arbitrary, if $\orb(c)$ is infinite, we can construct uncountably many different such points $x$.
	
	Since $f$ is zigzag-free, and thus every $f^n$ is zigzag-free (Lemma~\ref{lem:itzigzag}), then for every intervals $K,J$ such that $f^n(J)=K$ and $J$ minimal such (so there is no $J'\subset J$ such that $f^n(J')=K$), it holds that critical points of $f^n$ in $\Int(J)$ are mapped to $\partial K$.
	
	Assume that there are subcontinua $A,B\subset\underleftarrow{\lim}\{I,f\}$ such that $x\in A\cap B$, and such that $A\not\subset B$, $B\not\subset A$. Denote the projections by $A_i=\pi_i(A), B_i=\pi_i(B)$ for $i\geq 0$. Thus $x_i\in A_i\cap B_i$ for all $i\geq 0$. There exists $N\in\N$ such that $A_i\setminus B_i, B_i\setminus A_i\neq\emptyset$ for all $i\geq N$. By the definition of $x$ and since $L_j$ can be chosen to be arbitrarily small, there are $k,j\in\N$ such that $f^k(L_j)\subset A_N\cap B_N$. It follows that there is a critical point $c\in\Int(A_{N+k}\cup B_{N+k})$ which is mapped to $\Int(A_N\cup B_N)$, and that is impossible since $f^k$ is zigzag-free. So, $x$ is an endpoint.
\end{proof}

\begin{theorem}\label{thm:endpt}
	If $f$ is long-zigzag leo map, and there exists $c\in C$ such that $c\in\omega(c)$, then there exists an endpoint in $\underleftarrow{\lim}\{I,f\}$. If additionally $\orb(c)$ is infinite, then $\underleftarrow{\lim}\{I,f\}$ has uncountably many endpoints.
\end{theorem}
\begin{proof}
	The proof follows analogously as in the previous lemma. Since $f$ is leo, it follows by Observation~\ref{obs:leo} that we can take $M\in\N$ such that $|A_M\cup B_M|<\eps$, where $\eps$ is the constant from the definition of a long-zigzag map and then proceed analogously as in the proof of Lemma~\ref{lem:uncountably}.
\end{proof}

\begin{remark}
	In the proof of Lemma~\ref{lem:uncountably} we use in several places that $c\in \omega(c)$. It is natural to ask if the proof works in different setting.
	The previous lemma and theorem also hold true if we assume that $c\in\omega(C)$ for every $c\in C$ under the zigzag-free and under the long-zigzag leo assumptions respectively, if we do the following. On the first occurrence replace $c\in\omega(c)$ with $c\in\omega(C)$ and then take any $c'$ such that $f^k(c')\in K$ and continue with that $c'$. In the next step we would have "since $c'\in\omega(C)$" etc. To get uncountably many endpoints in that case one needs to assume that every neighbourhood of every critical point contains at least two points from $\orb(C)$ as one can see following the preceding proof.
\end{remark}

\begin{example}
	Besides some obvious examples of non-leo maps that give in the inverse limit countably infinitely many endpoints,
	there also exist leo maps $f\colon I\to I$ such that $\underleftarrow{\lim}\{I,f\}$ has countably infinitely many endpoints, see \eg Figure~\ref{fig:leoexample}.
	\begin{figure}
		\begin{tikzpicture}[scale=5]
		\draw[thick] (1,0) -- (1/2, 1) -- (5/16, 5/16) -- (1/4,1/2) -- (5/32, 5/32) -- (1/8,1/4)--(3/32,3/32);
		\draw (0,0) -- (0,1) -- (1,1) -- (1,0) -- (0,0);
		\draw (1/2,-0.02)--(1/2,0.02);
		\draw (5/16,-0.02)--(5/16,0.02);
		\draw (1/4,-0.02)--(1/4,0.02);
		\draw (5/32,-0.02)--(5/32,0.02);
		\draw (1/8,-0.02)--(1/8,0.02);
		\draw (-0.02,1/2)--(0.02,1/2);
		\draw (-0.02,1/4)--(0.02,1/4);
		\draw (-0.02,1/8)--(0.02,1/8);
		\draw[dashed] (0, 0) -- (1, 1);
		\node at (0.5,-0.1) {\small $\frac{1}{2}$};
		\node at (1/4,-0.1) {\small $\frac{1}{4}$};
		\node at (0.115,-0.1) {\small $\frac{1}{8}$};
		\node at (5/16,-0.1) {\small $\frac{5}{16}$};
		\node at (5/32+0.01,-0.1) {\small $\frac{5}{32}$};
		\node at (-0.1,1/2) {\small $\frac{1}{2}$};
		\node at (-0.1,1/4) {\small $\frac{1}{4}$};
		\node at (-0.1,1/8) {\small $\frac{1}{8}$};
		\node at (0.1,0.9) {\small $f$};
		\end{tikzpicture}
		\caption{Locally eventually onto interval map for which the inverse limit has countably infinitely many endpoints.}
		\label{fig:leoexample}
	\end{figure}
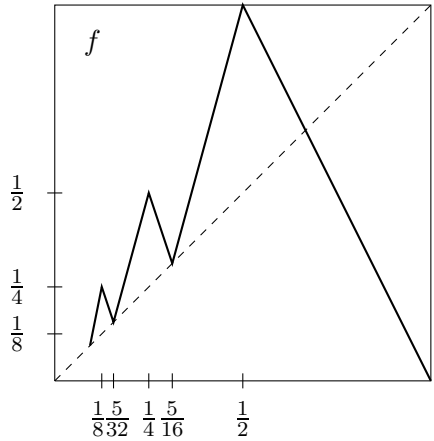
	Note that every $(x,x,x,\ldots)$, where $x\in I$ is a fixed point of $f$, is an endpoint (see \cite{BM}, Example 3 for the detailed argument), which gives countably infinitely many endpoints. Furthermore, all points $(x_0,x_1,x_2,\ldots)$ for which there exists $N\in \N$ so that $x_i\neq 1/2^n$ for all $i>N$ are contained in interiors of their $N$-basic arcs and are thus not endpoints. There is only countably many remaining points that can be folding points, they are of the form $\hat f^n((1/2,1/2^{2},1/2^3,\ldots))$ for every $n\in \Z$. 
\end{example}

\appendix

\section{Characterization of retractability in piecewise monotone interval maps (by Henk Bruin)}\label{appA}
This appendix is a natural (however more technical) continuation of Section~\ref{sec:persistence}.
In general, it is difficult to check if $f$ is retractable or non-retractable along $\omega(C)$, or to see what it has to do with recurrence. In the case of unimodal maps, \cite[Section 3]{Br-trans} gives different (non-equivalent) notions of recurrence of the critical point which can be checked using symbolic properties of kneading sequences. We relate the notion of a non-retractable map introduced in Section~\ref{sec:persistence} to the notions from \cite{Br-trans}, and we generalize them to the multimodal setting. 


\begin{definition}
	Given a piecewise monotone map $f\colon I\to I$, for $x\in I$ and $n\in\N$ we define
	$$r_n(x)=\min\{|f^n(x)-f^n(a)|, |f^n(x)-f^n(b)|\},$$ 
	$$R_n(x)=\max\{|f^n(x)-f^n(a)|, |f^n(x)-f^n(b)|\},$$
	where $[a,b] \ni x$ is a maximal interval such that $f^n|_{[a,b]}$ is monotone. 
	Note that if $x$ is a critical point of $f^n$, then there are two intervals $[a,b]$ we can choose, but regardless of this choice, 
	if $f^N(x)\in C$ for some $N\in\N$, then $r_n(x)=0$ for all $n>N$, so $r_n(x)$ is well-defined for every $x\in I$ and $n\in\N$. 
	For $R_n(x)$ we choose the maximum of the lengths $f^n([a,b])$ if there are two choices.
	We write $M_n(x)=f^n([a,b])$, taking again the maximal length interval if there are two choices.
\end{definition}

Given a map $f\colon I\to I$ and $c\in C$, let $\tilde c=f^k(c)$, 
where $k\in\N$ is the smallest such that $f^k(c)\in\Int(I)$. Note that $\tilde c$ might not be well defined if 
$f(0),f(1)\in\{0,1\}$, but if it is, then $\tilde c=f^k(c)$ for $k\leq 3$. In particular, 
if $f:[f^2(c),f(c)]\to [f^2(c),f(c)]$ is unimodal map restricted to its core, then $\tilde c=f(0)=f^2(1)=f^3(c)$. 
For simplicity, we will abuse the notation and denote the interval $M_{n+1-k}(\tilde c)$ by $M_n(f(c))$, 
and similarly $r_n(f(c)), R_n(f(c))$.

\begin{observation}
If $\omega(C)\subset\{0,1\}$, then $f$ is automatically non-retractable along $\omega(C)$, since no point of $\omega(C)$ is contained in an open interval of $I$.
\end{observation}

We call a map $f$ {\em closure non-retractable} if for every open interval $U$ there is $N \in \N$ such that 
$U$ can be pulled back along backward orbits $(x_n)_{n \geq 0}$ for at most
$N$ times, where $x_0 \in \overline{U}$. Otherwise $f$ is called {\em closure retractable}.
This is stronger than non-retractable, which has the same definition, except that $x_0$ must be in the interior of $U$ itself. 

\begin{lemma}\label{lem:closureper}
Let $f: I\to I$ is piecewise monotone map with critical set $C$.
Then $f$ is closure non-retractable if and only if $\limsup_{n \to \infty} R_n(f(c)) = 0$ for all $c\in C$.
\end{lemma}

\begin{proof}
	For both implications we argue by the contrapositive.\\
	($\Rightarrow$) Suppose there is $c \in C$ and $\delta > 0$ such that $R_n(f(c)) > \delta$ infinitely often.
	Hence there is a sequence $(n_k)$ such that $n_{k+1} - n_k \to \infty$,
	$f^{1+n_k}(c) \to x$ and $R_{n_k}(f(c)) > \delta/2$. In particular, $x \in \omega(C)$ and
	so is its forward orbit.
	Therefore, for at least one of $U = [x-\delta/2, x]$ or $U = [x, x+\delta/2]$,
	we can pull $U$ back along the orbit $f^{n_{k+1}-n_k}(x), f^{n_{k+1}-n_k-1}(x), \dots,
	x$ in a monotone way.
	Since $n_{k+1}-n_k \to \infty$, we can find an infinite path in $\omega(C)$ along which
	$U$ can be pulled back along an orbit in $\omega(C)$ indefinitely, contrary to $f$ being closure non-retractable.\\[2mm]
($\Leftarrow$)
	Suppose $f$ is closure retractable, so there is an interval $U$ and an infinite backward orbit
	inside $\omega(C)$, i.e.,
	$f(x_{n+1}) = x_n \in \omega(C)$ for $n \geq 0$ and $x \in U$,
	along which $U$ can be monotonically pulled back.
	We denote the pull-back intervals by $U_n$.
	Choose $n$ arbitrary and take $k \geq 0$ minimal such that $c_{k+1} \in U_n$ for some $c\in C$.
	Because $U \cap \omega(C)\neq\emptyset$, such $k$ can be found.
	Then $M_{k}(f(c)) \supset U_n$ and therefore $M_{n+k}(f(c)) \supset U$ and $R_{n+k}(f(c)) > |U|/2$.
	Since $n$ is arbitrary, we find $\limsup_n R_n(f(c)) > 0$.
\end{proof}

\begin{corollary}\label{cor:persis}
Let $f: I\to I$ be a piecewise monotone map with critical set $C$.
If $\limsup_{n \to \infty} R_n(f(c)) = 0$ for all $c\in C$, then $f$ is non-retractable along $\omega(C)$.
\end{corollary}

\begin{proof}
 Since closure non-retractable is stronger than non-retractable, this implication is immediate from Lemma~\ref{lem:closureper}.
\end{proof}

If $f\in C^2$ does not have attracting periodic orbits or wandering intervals (in particular, for piecewise monotone maps
with constant slope $> 1$), then the following {\em non-contraction property} holds:
\begin{quote}
There is $\delta_0$ such that if $J$ is a subinterval in $B(C, \delta_0) \setminus C$ 
and $j \geq 1$ is minimal such that $f^j(J) \cap C \neq \emptyset$, then $|f^j(J)| > 2 |J|$.
\end{quote}
See e.g.\ \cite{BRSS} for metric results in this direction. 
Without the $C^2$ smoothness assumption, the non-contraction property may fail, as
the following example shows:
$$
f:I \to I, \qquad f(x) = \begin{cases} 
                          2x(1-x) + \frac12, & \text{ if } x \in [0,\frac12], \\[2mm]
                          2(1-x)  & \text{ if } x \in [\frac12,1].
                         \end{cases}
$$
For this map, any left neighbourhood $J$ of $c = \frac12$ returns to a right neighbourhood
of $c$ after three iterations, and if $J$ is sufficiently small, then $|f^3(J)| < |J|$.

However, the non-contraction property holds for maps with piecewise constant slope $\geq 1$; 
this is easy to prove provided $f^j(C) \cap C = \emptyset$.
Since maps without wandering intervals or periodic intervals (other than the whole space $[0,1]$) 
are topologically conjugate to constant slope maps as above, and the nature of endpoint in inverse limit spaces 
is not affected by conjugacies, it is not very restrictive to make this assumption.
In particular, leo maps have neither periodic interval nor wandering intervals, and they have positive entropy.
Therefore they are conjugate to maps with the non-contraction property, and 
so Theorem~\ref{thm:per} below fits in the setting of Theorem~\ref{thm:FisE}.

\begin{lemma}\label{non-dense}
Assume that the piecewise monotone map $f:I \to I$ satisfies the non-contraction property.
If $r_n(f(c)) \to 0$ for every $c \in C$, then $\omega(C)$ is nowhere dense.
\end{lemma}

\begin{proof}
We can assume that $f$ has no periodic cycles of intervals other than the whole space, because otherwise we 
restrict $f$ to this cycle and consider the first return map to one of the intervals in the cycle.
 We can also assume that $f$ has positive topological entropy, because if $h_{top(f)} = 0$, 
 then all $\omega$-limit sets are finite or Cantor sets, see \cite{S,BS}.
 One characterisation of positive entropy is that $I$ contains horseshoes, 
 i.e., invariant compact sets $H$ such that $f^n:H \to H$ is a horseshoe map,
where $\frac1n h_{top}(f^n|H)$ is arbitrarily close to  $h_{top}(f)$, see \cite{MSz}.
 If $H \cap C \neq \emptyset$, then we can find a subset $H'$ bounded away from $C$, on which $f^{n'}$ acts as
 a horseshoe for some $n' \geq n$, so for every neighbourhood $U \owns c$, there is always a compact $f$-invariant set
 that intersects $U$ in a Cantor set but that is bounded away from $C$.

	Define $X_\delta = \{ x \in [0,1] : d(\orb(x), C) \geq \delta\}$ and 
	$X'_\delta \subset X_\delta$ is the perfect subset that the Cantor-Bendixson Theorem
	supplies\footnote{The Cantor-Bendixson Theorem \cite{Kechris} states that in a Polish space, every closed set $X$ is
		the disjoint union of a perfect set and a countable set (and this perfect set is nonempty if $X$ is uncountable).
		In our case, $X_\delta$ could have isolated points, for example if there is an orientation reversing periodic 
		point at distance $\delta$ from $C$, but we can simply remove them.},
		this latter set must contain a horseshoe, and in particular be nonempty, for small
	positive $\delta$. 
	
	Assume by contradiction that $\omega(C)$ contains intervals $J$. Then the interior of $\omega(C)$
	intersects $X'_\delta$ for $\delta$ sufficiently small. Indeed, if $J$ is such an interval, then
	$f^k(J) \owns c$ for some minimal $k$ and some $c \in C$.
	Therefore $f^k(J)$ intersects some horseshoe $H$, say with $d(H, C) = \delta$.
	But then there are points $x,y \in J$ such that $f^k(x), f^k(y) \in H$ 
	which implies that $x,y \in X'_\delta$.
	We additionally take $\delta < \delta_0/2$ for the $\delta_0$ in the non-contraction property.
	
	By our assumption that $r_n(f(c)) \to 0$ for all $c \in C$, there is $N \in \N$ such that $r_n(f(c)) < \delta/2$
	for all $n \geq N$ and $c \in C$. 
	Thus, since $X'_\delta$ has no isolated points, 
	there are distinct $x,y \in X'_\delta$ such that $[x,y] \subset \omega(C)$ and 
	$\cup_{j=0}^N f^j(C) \cap [x,y] = \emptyset$.
	
	Let $k \geq N$ minimal, $c \in C$ be such that $f^{1+k}(c) \in (x,y)$.
	Then $M_k(f(c)) \supset [x,y]$ and $M_{k'}(f(c)) \supset (c'-\delta, c'+\delta)$ for 
	$k' = \min\{ j \geq k : M_j(f(c)) \cap C \neq \emptyset\}$ and some $c' \in C$.
	
	Because $r_{k'}(f(c)) < \delta/2$, the component of $M_{k'}(f(c)) \setminus \{ f^{1+k'}(c) \}$
	intersecting $C$ has length $> \delta/2$ while the other component contains one of
	$f^{k'-k}(x)$ or $f^{k'-k}(y)$, say the first. Let $k'' = \min\{ j > k' : M_j(f(c)) \cap C \neq \emptyset\}$.
	By the non-contraction property, $|M_{k''}(f(c))| > \delta$, and we can repeat the argument.
	In particular, $d(f^{k''-k}(x), f^{1+k''}(c)) < r_{k''}(f(c)) < \delta/2$.
	But if we can do this indefinitely, $d(f^j(x), f^{1+k+j}(c)) < \delta/2$ for all $j \geq 0$ and
	$J := [x, f^{1+k}(c)]$ is in fact a wandering interval, and so $\omega(C)$ is disjoint 
	from $\cup_{n \geq 0} f^{-n}(J)$. If $I \setminus\cup_{n \geq 0} f^{-n}(J)$ contains an interval, then this interval
	must be preperiodic, so there is a periodic interval other than the whole space which was excluded at the beginning of the proof. Therefore, $\omega(C)$ is nowhere dense. 
\end{proof}

\begin{theorem}\label{thm:per}
Let $f: I\to I$ satisfy the non-contraction property.
The map $f$ is non-retractable along $\omega(C)$ if and only if $r_n(f(c)) \to 0$ for all $c\in C$. 
\end{theorem}

\begin{proof}
	($\Rightarrow$) This is basically the same proof as the first half of Lemma~\ref{lem:closureper}.
	Indeed, suppose there is $c \in C$ and $\delta > 0$ such that $r_n(f(c)) > \delta$ infinitely often.
	Hence there is a sequence $(n_k)$ such that $n_{k+1} - n_k \to \infty$,
	$f^{1+n_k}(c) \to x$ and $r_{n_k}(f(c)) > \delta/2$. In particular, $x \in \omega(C)$ and
	so is its forward orbit.
	Therefore, we can pull $U := (x-\delta/2,x+\delta/2)$,
	back along the orbit $f^{n_{k+1}-n_k}(x), f^{n_{k+1}-n_k-1}(x), \dots,
	x$ in a monotone way.
	Since $n_{k+1}-n_k \to \infty$, we can find an infinite path in $\omega(C)$ along which
	$U$ can be pulled back along an orbit in $\omega(C)$ indefinitely, contrary to $f$ being non-retractable.\\[2mm]
	($\Leftarrow$) Assume by contradiction that there is an open set $U$ that can be pulled back indefinitely
	along a backward orbit $(x_n)_{n \geq 0}$ and $x_0$ is an interior point of $U$.
	Let $U^-$ and $U^+$ be two closed subintervals of $U$, one on either side of $x_0$, and both disjoint
	from $\omega(C)$. Since $\omega(C)$ contains no intervals (by Lemma~\ref{non-dense}), $U^+$ and $U^-$ can indeed be found, 
	and making them smaller if necessary, we can assume that they are also disjoint from the 
	forward critical orbits. Let $\delta := \min\{ |U^-|, |U^+|\} > 0$, and let
     $U'$ be the component of $U \setminus (U^+ \cup U^-)$
	that contains $x_0$.
	
	For $n \in \N$ arbitrary, let $U_n$, $U'_n$, $U^-_n$ and $U^+_n$ be the $n$-th pullbacks of
	$U$, $U'$, $U^-$ and $U^+$ respectively. Since $x_n \in U'_n \cap \omega(C)$ there is $k \in \N$
	minimal such that $f^{1+k}(c) \in U'_n$ for some $c \in C$.
	But then $M_k(f(c)) \supset U^-_n \cup U'_n \cup U^+_n$ and $M_{k+n}(f(c)) \supset U^- \cup U' \cup U^+$,
	whence $r_{n+k}(f(c)) \geq \delta$.
	Since $n$ is arbitrary,  $r_n(f(c)) \not\to 0$.
\end{proof}

\section{Characterization of arcs in inverse limits of piecewise monotone maps}\label{appB}

The characterization of $L$-endpoints in $\underleftarrow{\lim}\{I, f\}$ (and furthermore, the characterization of $f$ for which every $L$-endpoint in $\underleftarrow{\lim}\{I, f\}$ is an endpoint) requires an answer to the following problem:

\begin{problem}
	Characterize $f_i\colon I\to I$ such that $\underleftarrow{\lim}\{I, f_i\}$ is an arc.
\end{problem}

In this section we assume that $f\colon I\to I$ is a continuous surjection with finitely many critical points $\{0=c_1, c_2, \ldots, c_{n-1}, c_n=1\}=:C$ and we only study inverse limits $\underleftarrow{\lim}\{I, f\}$ with a single bonding map $f$. Thus we answer the above problem partially. 

We say that a point $x\in I$ is {\em $f$-attracted} to $y\in I$ if $f^n(x)\to y$ as $n\to\infty$. 
The set of fixed points of $f$ will be denoted by $\Fix(f)=\{y\in I: f(y)=y\}$. First we prove two simple lemmas that we use later when establishing the mentioned characterization. 

\begin{lemma}\label{lem:archomeo}
	Let $g\colon I\to I$ be a homeomorphism such that $g(0)=0$, $g(1)=1$. Then every point of $I$ is $g$-attracted to a point in $\Fix(g)$. Moreover, every point in a connected component of $I\setminus \Fix(g)$ is $g$-attracted to the same point in $\Fix(g)$.
\end{lemma}
\begin{proof}
	First note that $\Fix(g)$ is closed, since $I$ is closed. So, for every $x\not\in \Fix(g)$ there exist $x_1, x_2\in \Fix(g)$ such that $x_1<x<x_2$ and $(x_1,x_2)\cap \Fix(g)=\emptyset$. Since $g$ is a homeomorphism, if $g(x)<x$, then $x_1<g(y)<y$ for all $y\in(x_1, x_2)$, so $g^n(y)\to x_1$. Analogously if $g(x)>x$ we conclude that $g^n(y)\to x_2$ for all $y\in (x_1, x_2)$.
\end{proof}

For the simplicity of phrasing the following proof we introduce the following notation: given a map $f\colon I\to I$ and $[a,b]\subset I$, 
we say that $f|_{[a,b]}$ is {\em virtually increasing} if $f([a,b])=[f(a),f(b)]$, $f(a)<f(b)$ and $f(x)\not\in\{f(a),f(b)\}$ for every $x\in(a,b)$.

\begin{lemma}\label{lem:surjint}
	Let $f\colon I\to I$ be a map. For every $[x,y]\subset I$ there exists an interval $[x',y']\subset I$ such that $f^2([x',y'])=[x,y]$ and $f^2|_{[x',y']}$ is virtually increasing.
\end{lemma}

\begin{proof}
	Fix $[x,y]\subset I$.
	Assume first there are $c<d\in I$ such that $f(c)=0, f(d)=1$. If there are more such points $c<d$, choose so that $[c,d]$ is minimal with $f(c)=0, f(d)=1$. Then define $a=\sup\{z\in [c,d]: f(z)=x, y\not\in f([0,z])\}$, and  
	$b=\inf\{z\in [c,d]: f(z)=y, x\not\in f([z,1])\}$. Since $f$ is a continuous surjection, $a<b$ are well defined,  $f([a,b])=[x,y]$,
	and $f|_{[a,b]}$ is virtually increasing. Using the same procedure we find $x',y'$ such that $f([x',y'])=[a,b]$, and $f|_{[x',y']}$ is virtually increasing, which finishes the proof in this case. 
	
	In the other case, due to surjectivity of $f$, note that there are $c'<d'\in I$ such that $f(c')=1, f(d')=0$, and thus also $c<d$ such that $f(c)=d'$, $f(d)=c'$. But then we have $f^2(c)=0$, $f^2(d)=1$ and the proof again follows analogously.
\end{proof}

A point $x\in K$ in a continuum $K$ is called a {\em cut point}, if $K\setminus \{x\}$ is not connected. Otherwise, $x$ is called a {\em non-cut} point.
We will use the following topological characterisation of arc. A {\em topological ray} is a continuous one-to-one image of $[0,1)$.

\begin{proposition}[Theorem 6.17 in \cite{Na}]\label{prop:Nadarcchar}
	A continuum $K$ is an arc if and only if $K$ has exactly two non-cut points.
\end{proposition} 

Furthermore, we will also use the following well-known theorem of Bennett.

\begin{proposition}[\cite{Ben} and Theorem 19 in \cite{InMah}]\label{prop:Ben}
	Suppose $f:I\to I$ and there is a point $\zeta\in (0,1)$ so that:
	\begin{enumerate}
		\item $f([\zeta, 1])\subseteq [\zeta, 1]$,
		\item $f|_{[0, \zeta]}$ is monotone,
		\item There is $j\in \N$ such that $f^j([0, \zeta])=I$.
	\end{enumerate}
	Then $\underleftarrow{\lim}(I, f)$ is the closure of a topological ray having for a remainder the continuum
	$\underleftarrow{\lim}([\zeta,1], f|_{[\zeta,1]})$.
\end{proposition}

\begin{lemma}\label{lem:arc1}
	Assume that $f\colon I\to I$ is a piecewise monotone map such that $f(0)=0$, $f(1)=1$, $\Fix(f)=\{0,1\}$, and every point $x\in(0,1)$ is $f$-attracted to $1$. Then $\underleftarrow{\lim}(I, f)$ is an arc.
\end{lemma}
\begin{proof}
	We will show that $(0,0, \ldots)$ and $(1,1,\ldots)$ are the only two non-cut points of $\underleftarrow{\lim}(I, f)$. From Proposition~\ref{prop:Nadarcchar} it then follows that $\underleftarrow{\lim}(I, f)$ is an arc.
	
	Let $c>0$ be the minimal critical point of $f$ different than $0$ or $1$. Then $f|_{[0,c]}$ is one-to-one and since $f^n(c)\to 1$ it follows that $\cup_{n\in\N}f^n([0,c])\supset [0,1)$. Moreover, for every $x>0$ we have $f([x,1])\subsetneq [x,1]$. Thus $f([c,1])\subset [c,1]$ and $\cap_{n\in\N}f^n([c,1])=\{1\}$. At this point we could use a slight extension of Proposition~\ref{prop:Ben} to conclude that $\underleftarrow{\lim}(I, f)$ is a topological ray which limits on a point $\underleftarrow{\lim}([c,1], f|_{[c,1]})=\{(1,1,\ldots)\}$, but we will prove the theorem directly for completeness. Take any $x=(x_0,x_1, \ldots)\in\underleftarrow{\lim}(I, f)$ such that $x\neq (0,0,\ldots),(1,1, \ldots)$. Then there exists $N\geq 0$ such that $x_i\in (0,c)$ for all $i\geq N$. Then for all $i>N$ we have $f([0,x_i])=[0,x_{i-1}], f([x_i,1])=[x_{i-1},1]$ and thus $A=\underleftarrow{\lim}([0,x_i], f|_{[0,x_i]})$ and $B=\underleftarrow{\lim}([x_i, 1], f|_{[x_i,1]})$ are well-defined, non-degenerate subcontinua, such that $A\cup B=\underleftarrow{\lim}\{I,f\}$ and $A\cap B=\{x\}$, so $x$ is a cut point.
\end{proof}

\begin{lemma}\label{lem:arc2}
	Assume that $f\colon I\to I$ is a piecewise monotone map such that there exists $d\in(0,1)$ such that $\Fix(f)=\{0,d,1\}$. In addition assume that every $x\in (0,1)$ is $f$-attracted to $d$. Then $\underleftarrow{\lim}(I, f)$ is an arc.
\end{lemma}
\begin{proof}
	Note that $f(x)>x$ for all $x\in (0,d)$ and $f(x)<x$ for all $x\in (d,1)$. Let $c_1>0$ be the smallest and $c_2<1$ be the largest critical point of $f$. Then $f([0,c_2])\subset [0,c_2]$ and $f([c_1, 1])\subset [c_1,1]$, so $A=\underleftarrow{\lim}([0,c_2], f|_{[0,c_2]})$ and $B=\underleftarrow{\lim}([c_1,1], f|_{[c_1,1]})$ are well-defined continua. Also, $\underleftarrow{\lim}(I, f)=A\cup B$ and since $f([c_1,c_2])\subset [c_1, c_2]$ we have that $A\cap B=\underleftarrow{\lim}([c_1, c_2], f|_{[c_1,c_2]})$. Moreover, since $f([a,b])\subsetneq [a,b]$ for every $0<a<d<b<1$, we conclude that $\cap_{n\in\N}f^n([c_1, c_2])=\{(d,d,\ldots)\}$. So applying an extension of Proposition~\ref{prop:Ben} (or arguing analogously as in the previous proof) to both $A$ and $B$ we conclude that $\underleftarrow{\lim}(I, f)$ consists of two arcs 
	which are joined at their endpoint, thus it is an arc.
\end{proof}

\begin{remark}
	Note that the previous two Lemmas (and consequently the following theorem) can be extended to maps $f$ which can have infinitely many critical points, but for which there exist neighbourhoods of repelling fixed points ($0$ in the first lemma and $0,1$ in the second) which do not contain critical points. This is for instance not the case in the Henderson's map \cite{Henderson} which gives the pseudo-arc in the inverse limit.
\end{remark}

\begin{theorem}\footnote{In the late stages of writing this paper we found out through personal communication with Sina Greenwood that they obtained (with Sonja \v Stimac) a characterization of arc which works for all continuous interval maps $f$.}
	Let $f\colon I\to I$ be a piecewise monotone map. The space $\underleftarrow{\lim}(I,f)$ is an arc if and only if for every $x\in I$ there exists $y\in \Fix(f^2)$ such that $x$ is $f^2$-attracted to $y$, and such that all points in a connected component of $I\setminus \Fix(f^2)$ are $f^2$-attracted to the same point in $\Fix(f^2)$.
\end{theorem}
\begin{proof} ($\Rightarrow$)
	Assume $\underleftarrow{\lim}(I, f)$ is an arc. Then $\hat f^2\colon\underleftarrow{\lim}(I, f)\to \underleftarrow{\lim}(I, f)$ is an arc self-homeomorphism which preserves the endpoints of $\underleftarrow{\lim}(I, f)=\underleftarrow{\lim}(I, f^2)$. Note that $\hat{f}^2$ is necessarily an orientation preserving homeomorphism, so $\hat{f}^2(0)=0$ and $\hat{f}^2(1)=1$. From Lemma~\ref{lem:archomeo} we conclude that every $x=(x_0, x_1, \ldots)\in\underleftarrow{\lim}(I, f^2)$ is attracted to a fixed point of $\hat f^2$, \ie there is a point $\gamma\in \Fix(f^2)$ such that $(f^{2n}(x_0), f^{2n}(x_1), \ldots)\to (\gamma,\gamma, \ldots)$ as $n\to\infty$ for $(f^{2n}(x_0), f^{2n}(x_1),\ldots), (\gamma,\gamma, \ldots)\in \underleftarrow{\lim}(I, f^2)$. It follows that $f^{2n}(x_0)\to \gamma\in \Fix(f^2)$ as $n\to \infty$ for an arbitrary $x_0\in I$.\\
	Now take $x_0,y_0\in I$ such that $[x_0,y_0]\cap \Fix(f^2)=\emptyset$, and assume for contradiction that there exist $a\neq b\in \Fix(f^2)$ such that $f^{2n}(x_0)\to a$ and $f^{2n}(y_0)\to b$ as $n\to\infty$. By Lemma~\ref{lem:surjint}, for every $i\in\N$ we can inductively find  $x_i, y_i\in I$ such that $f^2([x_i,y_i])=[x_{i-1},y_{i-1}]$ and $f^2|_{[x_i,y_i]}$ is virtually increasing. Then, obviously $\underleftarrow{\lim}([x_i,y_i], f^2|_{[x_i, y_i]})$ is a subcontinuum of $\underleftarrow{\lim}(I, f)$ which is an arc, and thus $\underleftarrow{\lim}([x_i,y_i], f^2|_{[x_i, y_i]})$ is an arc with endpoints $x=(x_0,x_1, \ldots)$, and $y=(y_0, y_1,\ldots)$. Note that $\hat f^{2n}(x)\to (a,a,\ldots), \hat f^{2n}(y)\to (b,b,\ldots)$ as $n\to \infty$. Thus Lemma~\ref{lem:archomeo} implies that there is a point $(z,z,\ldots)\in \underleftarrow{\lim}([x_i,y_i], f^2|_{[x_i, y_i]})$. Specifically, there exists a point $z\in \Fix(f^2)$ such that $z\in [x_0,y_0]$, which is a contradiction.
	
	($\Leftarrow$) To prove the other direction, assume that the critical set of $f$ is finite, and assume that every point is $f^2$-attracted to a point in $\Fix(f^2)$ such that every point in a connected component of $I\setminus \Fix(f^2)$ is attracted to the same point. Note that $I\setminus \Fix(f^2)$ is a union of intervals $U_i=(a_i,b_i)$, $0\leq i\leq N$, such that either $f^{2n}(x)\to a_i$ or $f^{2n}(x)\to b_i$ for every $x\in (a_i, b_i)$. Assume without loss of generality that $f^{2n}(x)\to b_i$ for every $x\in (a_i, b_i)$. Then $f^2(x)\geq a_i$, for otherwise there would exist a point in $(a_i, b_i)$ which is $f^2$-attracted to some fixed point smaller or equal to $a_i$. If there exists no point $x\in (a_i,b_i)$ so that $f^2(x)>b_i$, then $f^2|_{(a_i,b_i)}$ is a homeomorphism and we are done. So, assume that there is $x\in (a_i, b_i)$ such that $f^2(x)>b_i$. Then (unless $b_i$ is the largest fixed point of $f^2$) there is a point $d_i\in \Fix(f^2)$ such that $(b_i, d_i)\cap \Fix(f^2)=\emptyset$ and $f^{2n}(y)\to b_i$ for all $y\in (b_i, d_i)$. Also, then $f^2(y)\leq d_i$ for all $y\in (b_i, d_i)$ (in the special case when $b_i$ is the largest fixed point of $f^2$ we take $d_i=1$). So we have two possibilities: either $f^2([a_i, b_i])=[a_i, b_i]$, or $f^2([a_i, d_i])=[a_i, d_i]$. Lemma~\ref{lem:arc1} implies that in the first case $\underleftarrow{\lim}([a_i, b_i], f^2|_{[a_i, b_i]})$ is an arc, and Lemma~\ref{lem:arc2} implies that in the second case $\underleftarrow{\lim}([a_i, d_i], f^2|_{[a_i, d_i]})$ is an arc (special case when $d_i=1$ is treated similarly as $f|_{[0,c_2]}$ in the proof of Lemma~\ref{lem:arc2} to conclude that the inverse limit is an arc). Therefore, $\underleftarrow{\lim}(I, f^2)$ can be obtained as a union of finitely many  arcs which pairwise intersect in their endpoints. We conclude that $\underleftarrow{\lim}(I, f^2)$ is an arc.
\end{proof}

\begin{remark}
	Note that the first direction of the previous proof indeed works for general $f$, not only for piecewise monotone. However, in the other direction the assumption of finiteness of the critical set is crucial. For example, Lemma~\ref{lem:arc1} fails in the case of the Henderson's map from \cite{Henderson} since it gives pseudo-arc as the inverse limit.
\end{remark}


\end{document}